\documentclass[11pt]{amsart}

\usepackage[margin=1in,includehead=true]{geometry}
\usepackage{amsmath,amssymb,amsthm}
\usepackage[hidelinks]{hyperref}

\newcommand{\Aut}{\operatorname{Aut}}
\newcommand{\Sym}{\operatorname{Sym}}
\newcommand{\Ded}{\operatorname{Ded}}
\DeclareMathOperator{\Zent}{Z}
\DeclareMathOperator{\supp}{supp}

\DeclareMathOperator{\cf}{cf}

\theoremstyle{plain}
\newtheorem{theorem}{Theorem}[section]
\newtheorem{corollary}{Corollary}[section]
\newtheorem{lemma}{Lemma}[section]
\newtheorem{proposition}{Proposition}[section]

\theoremstyle{definition}
\newtheorem{definition}{Definition}[section]
\newtheorem{question}{Question}[section]
\newtheorem{example}{Example}[section]
\theoremstyle{remark}
\newtheorem{remark}{Remark}[section]

\title[Separability of Subsets in Infinite Groups]%
      {Separability of Subsets in Infinite Groups}
\author{Suhui Wan}
\address{School of Mathematics and Statistics, Hainan University}
\email{20243005974@hainanu.edu.cn}
\keywords{separation of subsets, witness sets, conflict graphs, independent sets, infinite groups, Ramsey theory, cardinal arithmetic}
\subjclass[2020]{Primary 20E45; Secondary 03E17, 05C63, 05C69, 05D10, 20B35.}

\begin{document}

\maketitle

\begin{abstract}
We study the separability of subsets of an infinite group. Given an infinite group $G$ and subsets $A,B\subset G$, we say that $A$ and $B$ are \emph{separated} in $G$ if there exists an infinite symmetric subset $X\subset G$ with $e\in X$ and $XAX\cap B=\varnothing$. Question~17.102 of the \emph{Kourovka Notebook} raises a natural cardinal question: if $A$ and $B$ are disjoint and $|A|,|B|<|G|$, must they be separated? We first give a negative answer, constructing two essentially different families of counterexamples, based respectively on rigid binary relations and on torsion-free groups whose square sets have cardinality smaller than that of the group. On this basis we introduce the $W$-witness set $W_G(A,B)=\{g\in G:\{g,g^{-1}\}A\{g,g^{-1}\}\cap B=\varnothing\}$, obtain a necessary and a sufficient condition for separability given by its size, and study the critical range. The central result of this paper is a necessary and sufficient characterization of separability: taking the inverse pairs $\{g,g^{-1}\}$ as vertices, we construct the conflict graph $\Gamma_{A,B}$, and $A$ and $B$ are separated if and only if the associated graph contains an infinite independent set; separability is thereby converted into an ordinary graph-theoretic problem. From this we further derive the case of finite subsets, the case of Abelian groups, and several cardinal criteria, and we obtain, via Ramsey's theorem, a structural dichotomy for the non-separable case. In addition, we study the range of cardinalities of counterexample groups and the possible sizes of $W$-witness sets, and prove that no uniform necessary and sufficient criterion can depend only on $|A|$, $|B|$ and structural invariants of the group, independently of the specific position of $A$ and $B$.
\end{abstract}

\section{Introduction}\label{sec:intro}

The combinatorial study of subsets of groups concerns the interaction between the algebraic structure of a group and the behavior of its subsets under multiplication. Among the basic notions of this area are large, small, thick, and thin subsets, together with their cardinal versions. These notions provide an important framework for describing the combinatorial ``size'' of subsets in a group. A systematic account of this theory was given by Protasov \cite{protasov2011}.

Let $G$ be an infinite group and let $\kappa\le |G|$ be an infinite cardinal. A subset $A\subset G$ is called \emph{left $\kappa$-large} if there exists $F\subset G$ with $|F|<\kappa$ such that
\[
   G=FA.
\]
It is called \emph{left $\kappa$-small} if, for every left $\kappa$-large subset $L\subset G$, the set $L\setminus A$ is again left $\kappa$-large. A subset $A$ is called \emph{left $\kappa$-thick} if for every $F\subset G$ with $|F|<\kappa$ there exists $a\in A$ such that
\[
   Fa\subset A.
\]
The corresponding right and two-sided notions are defined symmetrically. Together with the related classes of thin, sparse, and prethick subsets, these notions form the basic terminology of the Subset Combinatorics of Groups \cite{protasov2011,partitions2015}.

One fundamental direction of the theory is the study of finite decompositions
\[
   G=A_{1}\cup\cdots\cup A_{n}
\]
and of the combinatorial properties that at least one cell of such a decomposition must possess. For instance, Protasov proved that every infinite group can be decomposed into countably many large subsets \cite{protasov2003}. When the cardinality of the group is taken into account, the decomposition problem displays considerably richer structure. Protasov and Slobodianiuk developed, by means of filtrations and related combinatorial methods, a systematic theory of decompositions into large and thick subsets \cite{partitions2015}; in particular, regular and singular cardinals behave markedly differently in such problems \cite{protasov2015note}.

The theory has since developed in several directions. The survey of Protasov and Protasova \cite{protasovSurvey} describes, among other topics, dynamical and descriptive characterizations of combinatorial size, Ramsey-product subsets, recurrence in $G$-spaces, and ultrafilter methods based on the Stone--\v{C}ech compactification of groups. These developments provide a broader framework for studying the interaction between cardinality and the multiplicative structure of subsets.

These studies reveal a close and subtle connection between the cardinal size of a subset and its multiplicative structure. For example, when a group is covered by finitely many subsets, the overall structure of the group tends to force one cell to be ``large'' in some multiplicative sense. In the opposite direction, an equally natural question arises: if two subsets are disjoint, can we use the multiplicative structure of the group to keep them apart, in some sense, by an infinite margin? This no longer asks for one set to cover the whole group, but for an infinite family of multipliers under whose two-sided action one set consistently avoids the other. It is in this context that we consider the following separation property.

\begin{definition}\label{def:sep}
Let $A,B$ be subsets of an infinite group $G$. We say that $A$ and $B$ are \emph{separated} in $G$ if there exists a subset $X\subset G$ satisfying
\begin{enumerate}
\item[(1)] $X$ is infinite;
\item[(2)] $e\in X$;
\item[(3)] $X=X^{-1}$, that is, $x\in X\Rightarrow x^{-1}\in X$;
\item[(4)] $XAX\cap B=\varnothing$.
\end{enumerate}
In this case, $X$ is called a \emph{witness set} for $A$ and $B$.
\end{definition}

This condition is by no means mere set-theoretic disjointness. Since $e\in X$, we have $A\subset XAX$, so separation requires the entire two-sided multiplicative neighborhood generated by $X$ to avoid $B$. Moreover, since $X=X^{-1}$, for every $g\in X$ we have $\{g,g^{-1}\}A\{g,g^{-1}\}\subset XAX$. Separability thus depends both on the cardinalities of $A,B$ and on their concrete position within the multiplicative structure of the group. This gives rise to several basic questions: which pairs of subsets can be separated? What algebraic or cardinal conditions cause non-separability? To what extent can one obtain uniform sufficient conditions, necessary conditions, and ultimately a necessary and sufficient condition?

This separation property is closely related to several problems in the \emph{Kourovka Notebook}. For instance, Problem~13.44 and the subsequent work study covering properties of the group by products such as $A_iA_i^{-1}$ arising from a finite decomposition, and by finitely many of their translates \cite{banakh2014}, while Problem~13.45 investigates whether, under cardinality restrictions, two cells of a decomposition can simultaneously fail to be large \cite{kourovka}. These problems reflect, from different angles, the interplay between the cardinality and the multiplicative structure of subsets. Against this background, Question~17.102 of the \emph{Kourovka Notebook} poses the following direct cardinal question about the separation property defined above.

\begin{question}[Kourovka Notebook \cite{kourovka}, 17.102]\label{q:sep}
Let $G$ be an infinite group and let $A,B$ be disjoint subsets of $G$ satisfying $|A|<|G|$ and $|B|<|G|$. Must $A$ and $B$ be separated in $G$?
\end{question}

Question~17.102 can be viewed as a cardinal version of the general separation problem above: it asks whether the cardinality condition ``strictly smaller than the whole group'' suffices to keep two disjoint subsets infinitely far apart. The answer is known to be affirmative when $A$ and $B$ are both finite, and when $G$ is Abelian \cite{kourovka}; these results can also be obtained, respectively, from the Banakh--Protasov result on the non-discrete Zariski topology and from an elementary recursive argument, and we will give the corresponding proofs in Theorem~\ref{thm:finite} and Theorem~\ref{thm:abel}. Thus the genuine difficulty of the question lies in general non-Abelian groups, where the concrete multiplicative positions of $A$ and $B$ begin to matter.

We first prove that the answer to Question~17.102 in general is negative.

\begin{theorem}\label{thm:sep}
There exist an infinite group $G$ and disjoint subsets $A,B\subset G$ with
\[
   |A|<|G|
   \qquad\text{and}\qquad
   |B|<|G|
\]
such that $A$ and $B$ are not separated in $G$.
\end{theorem}

The significance of this counterexample is not confined to refuting Question~17.102. It shows, first of all, that the mere cardinality conditions $|A|,|B|<|G|$ do not control separability. More importantly, it prompts us to move from the question ``does 17.102 hold?'' to the more general question ``what exactly determines separability?'': for given $G,A,B$, can we find a criterion that is both necessary and sufficient? To what extent can the cardinalities of non-separable configurations be pushed? Do the counterexamples obtained by different constructions share a common mechanism? The purpose of this paper is precisely to address these questions, thereby placing the negative answer to Question~17.102 within a more general theory of separability.

To illustrate this direction of research, we first examine two different kinds of counterexamples. The first family, presented in Section~\ref{sec:rigid}, is built on rigid binary relations. In $\Sym(\mathbb N)$, using a binary relation whose automorphism group is trivial, the transpositions lying in the relation and in its complement are collected into $A$ and $B$, respectively. Rigidity ensures that conjugation by any nontrivial multiplier sends some transposition of $A$ into $B$, thereby excluding the existence of an infinite witness set. This construction extends to arbitrary rigid binary relations, giving Theorem~\ref{thm:l1} and Theorem~\ref{thm:l2}.

Section~\ref{sec:square} presents a second, entirely different kind of counterexample. It no longer exploits the rigidity of a binary relation, but a special cardinal property of the square set of the group itself. Using a torsion-free group constructed by Newelski, one obtains a group of cardinality $2^{\aleph_0}$ whose square set is countable, and on it a non-separable pair is constructed. This construction further extends to groups satisfying $|G^2|<|G|$ with $G[2]$ finite, giving Corollary~\ref{cor:l3} and Theorem~\ref{thm:converse}.

These two families of counterexamples exploit, respectively, the rigidity of the conjugation action and the special cardinality of the square set; on the surface they share no obvious common structure. Hence, if one really wants to study separability, merely searching for counterexamples is not enough: we need an object that simultaneously contains all witness sets and directly reflects whether a given multiplier is ``eligible'' to enter one. Section~\ref{sec:W} starts precisely from this point.

For given $A,B\subset G$, define $W_G(A,B)=\{g\in G:\{g,g^{-1}\}A\{g,g^{-1}\}\cap B=\varnothing\}$. This set records all multipliers that are individually compatible with the separation requirement, and we call it the $W$-witness set of $A$ and $B$. It need not itself be a witness set, but every witness set is contained in it. This yields a basic necessary condition: if $W_G(A,B)$ is finite, then $A$ and $B$ cannot be separated. Conversely, letting $\lambda=\max\{|A|,|B|,\aleph_0\}$, if $|W_G(A,B)|>\lambda$, then a countable witness set can be obtained by successive selection. These results are given by Lemma~\ref{lem:necessary} and Lemma~\ref{lem:greedy}, respectively.

These two estimates also reveal the common mechanism underlying the preceding counterexamples: although rigidity of a binary relation and smallness of the square set are achieved in completely different ways, both compress $W_G(A,B)$ into a finite set. More generally, Section~\ref{sec:alt} studies the case where $|W|$ and $\lambda$ are of the same order of magnitude, and gives two examples with $|W|=\aleph_0=\lambda$ that are still not separable. This shows that within the critical range $\aleph_0\le |W|\le\lambda$, comparing the cardinalities of $W$ and $A,B$ no longer suffices to decide separability; one must further consider whether different multipliers can be accommodated simultaneously in a single witness set.

Section~\ref{sec:further} investigates several further cardinal and structural questions on this basis. On the one hand, the case of finite subsets and the case of Abelian groups provide two important positive results. On the other hand, using Newelski's cardinal results on square sets, one can determine the range of cardinalities of counterexample groups together with its generalizations, and study the possible sizes of $W$ in the non-separable case. These results show that separability can be determined neither by the sizes of $|A|,|B|$ alone nor by the cardinality of the whole group. Moreover, Theorem~\ref{thm:indis} shows that even when the group $G$ and the cardinalities of $A,B$ coincide completely, one pair may be separable while the other is not. Consequently, any intrinsic criterion that completely ignores the concrete position of $A,B$ in $G$ cannot serve as a necessary and sufficient condition in the general situation.

The above results ultimately lead to the core question of this paper: if the membership of a single multiplier in $W$ still does not decide separability, what really needs to be studied is the ``compatibility'' between different multipliers. Section~\ref{sec:conflict} converts this compatibility into a graph. Taking the inverse pairs $\{g,g^{-1}\}$ as vertices, defining the conflict between two vertices according to whether the two inverse pairs can appear simultaneously in some witness set, and adding a loop at each vertex outside $W$, we obtain the conflict graph $\Gamma_{A,B}$ of $A$ and $B$; see Definition~\ref{def:conflict}. The definition of a witness set then translates item by item into graph-theoretic language: $A$ and $B$ are separated if and only if $\Gamma_{A,B}$ contains an infinite independent set containing the vertex $\overline e$ corresponding to the identity element (Theorem~\ref{thm:indep}). After further excluding the looped vertices and the neighbors of $\overline e$, this condition is equivalent to the corresponding induced subgraph containing an infinite independent set (Theorem~\ref{thm:indep-S}). Separability is thereby finally reduced to an ordinary infinite graph-theoretic problem that no longer depends on a marked vertex.

This characterization not only provides a necessary and sufficient condition for separability, but also unifies several of the preceding results. The vertex degrees of the conflict graph are bounded by $4|A||B|$ (Lemma~\ref{lem:deg}), so the previous greedy selection method and the case of finite subsets both follow as direct corollaries of the graph-theoretic characterization. The characterization also yields several sufficient conditions for non-separability (Corollary~\ref{cor:lock}, Corollary~\ref{cor:unilateral}) and a sufficient condition for separability phrased in terms of the chromatic number (Corollary~\ref{cor:chromatic}). Combined with Ramsey's theorem, one can further describe the structure of the non-separable case: when $W$ is infinite and no infinite independent set exists, either all but finitely many eligible vertices are adjacent to $\overline e$, or infinitely many eligible vertices are pairwise adjacent; this gives a structural dichotomy for the non-separable case (Theorem~\ref{thm:ramsey}).

In summary, this paper does not treat Question~17.102 as an isolated conjecture, but regards it as an entry point into the study of separability of subsets of infinite groups. The counterexample to 17.102 reveals the limitation of the cardinality conditions; the $W$-witness set reveals the local constraints on individual multipliers; and the conflict graph unifies the compatibility between these local constraints. The main line of the paper thus gradually shifts from the construction of counterexamples to a structural theory of separability, culminating in a general necessary and sufficient characterization. Section~\ref{sec:prelim} collects the preliminary results needed throughout the paper; Sections~\ref{sec:rigid} and~\ref{sec:square} present the two families of counterexamples; Section~\ref{sec:W} develops the basic theory of $W$-witness sets; Section~\ref{sec:alt} studies the critical cardinal range; Section~\ref{sec:further} discusses several further cardinal and structural results; Section~\ref{sec:conflict} establishes the conflict graph and proves the necessary and sufficient characterization; finally, Section~\ref{sec:open} poses several open questions naturally arising from these results.

\section{Preliminaries}\label{sec:prelim}

All set-theoretic arguments in this paper are carried out in ZFC. We write $|S|$ for the cardinality of a set $S$.
Groups are written multiplicatively, $e$ denotes the identity element, and torsion-free groups are understood to be non-trivial. For a group $G$, let $G^2 = \{ g^2 : g \in G \}$ denote the square set of $G$, let $G[2] = \{ g \in G : g^2 = e \}$ denote the $2$-torsion subgroup of $G$, let $\Zent(G)$ denote the center of $G$, and let $C_G(g) = \{ x \in G : xg = gx \}$ denote the centralizer of an element $g\in G$.

\begin{lemma}[Newelski \cite{Newelski1987}, Theorem 2.3]
\label{lem:newelski-2.3}
There exists a group $G$ satisfying
\[
  |G| = 2^{\aleph_0},\qquad |G^2| = \aleph_0,
\]
and $G$ has no elements of order $2$.
\end{lemma}

\begin{lemma}[Newelski \cite{Newelski1987}, Remark 1]
\label{lem:newelski-torsionfree}
The group $G$ in Lemma~\ref{lem:newelski-2.3} is torsion-free.
\end{lemma}

\begin{corollary}\label{thm:newelski}
There exists a group $H$ satisfying
\[
  |H| = 2^{\aleph_0}, \qquad |H^2| = \aleph_0, \qquad H[2] = \{e\}.
\]
\end{corollary}

\begin{proof}
This is a direct consequence of Lemma~\ref{lem:newelski-2.3} and Lemma~\ref{lem:newelski-torsionfree}. In fact, $H$ may be taken to be a torsion-free group.
\end{proof}

\begin{lemma}[Newelski \cite{Newelski1987}, Theorem 1.3]
\label{lem:newelski-1.3}
Assume that $G$ is an infinite group. Then $|G| \le 2^{|G^2|}$.
\end{lemma}

\begin{remark}
\label{rem:newelski-1.3}
By Lemma~\ref{lem:newelski-1.3}, every torsion-free group with countable square set has cardinality at most $2^{\aleph_0}$, while the group constructed in Corollary~\ref{thm:newelski} has cardinality exactly $2^{\aleph_0}$ and a countable square set. Hence this group attains the maximal cardinality of a torsion-free group with countable square set.
\end{remark}

\begin{lemma}[Banakh--Protasov \cite{zariski2010}, Theorem 1]\label{lem:zariski-thm1}
For every infinite group $G$, the 2-nd Zariski topology $\mathcal Z_{G^2[x]}$ is not discrete.
\end{lemma}

\begin{lemma}[Banakh--Protasov \cite{zariski2010}, Theorem 5]\label{lem:zariski-thm}
Let the $n$-th Zariski topology $\mathcal Z_{G^n[x]}$ on a group $G$ be non-discrete. Then for any finite subsets $A_0,A_1,\ldots,A_n\subset G$ with $1\notin A_0A_1\cdots A_n$, there exists an infinite symmetric subset $X=X^{-1}\ni 1$ such that
\[
1\notin A_0XA_1X\cdots XA_n.
\]
\end{lemma}

\begin{corollary}\label{lem:zariski}
Let $G$ be an infinite group and let $F_0,F_1,F_2\subset G$ be finite sets with $e\notin F_0F_1F_2$. Then there exists an infinite set $X\subset G$ such that
\[
e\in X,\qquad X=X^{-1},\qquad e\notin F_0XF_1XF_2.
\]
\end{corollary}

\begin{proof}
By Lemma~\ref{lem:zariski-thm1}, $\mathcal Z_{G^2[x]}$ is non-discrete. Applying Lemma~\ref{lem:zariski-thm} with $n=2$ and $A_i=F_i$, $i=0,1,2$, yields the conclusion.
\end{proof}

\begin{lemma}[Newelski \cite{Newelski1987}, Corollary 2.5]\label{lem:newelski-2.5}
Let $\aleph_0\le\lambda\le\mu<\Ded\lambda$. Then there exists a torsion-free group $G$ satisfying
\[
  |G|=\mu,\qquad |G^2|=\lambda.
\]
\end{lemma}

\begin{theorem}[Ramsey \cite{Ramsey1930}, Theorem A]\label{thm:ramsey-A}
Let $\Gamma$ be an infinite class, and let $\mu$ and $r$ be positive integers; and let all those sub-classes of $\Gamma$ which have exactly $r$ members, or, as we may say, let all $r$-combinations of the members of $\Gamma$ be divided in any manner into $\mu$ mutually exclusive classes $C_i$ ($i=1,2,\ldots,\mu$), so that every $r$-combination is a member of one and only one $C_i$; then, assuming the axiom of selections, $\Gamma$ must contain an infinite sub-class $\Delta$ such that all the $r$-combinations of the members of $\Delta$ belong to the same $C_i$.
\end{theorem}

\section{The First Family of Counterexamples: Rigid Binary Relations}\label{sec:rigid}

Of course, at the outset the answer is not known to us, so we start from Question 17.102 itself. How should one approach it? As mentioned above, Protasov pointed out that the answer is affirmative when $A$ and $B$ are finite subsets or when $G$ is Abelian, which supports the expectation that the assertion holds in general. But a direct proof of the positive assertion tends to be difficult, and we may instead try some simple examples.

Take a relatively simple group $G=\Sym(\mathbb{N})$, whose cardinality is $2^{\aleph_0}$. We look for a pair $A,B$ of disjoint subsets with finite $W$; a natural approach is to let $A\sqcup B$ be the set of all transpositions of $G$. Since $\mathbb{N}$ is involved, it is preferable to choose subsets determined by the magnitude of the numbers, so that $A$ and $B$ are easy to tell apart. So we consider a graph $\Gamma=(V, E)$ with $V = \mathbb{N}$ and $E = \bigl\{\{n,m\} : 0 < n < m \le 2n\bigr\}$, and define accordingly
\[
A=\{(u\;v)\in G:\{u,v\}\in E\},\qquad
B=\{(u\;v)\in G:\{u,v\}\notin E\},
\]
where $(u\;v)$ denotes the transposition swapping $u$ and $v$.

The upper bound $2n$ in the edge set is chosen for the following reason: for each $n$, the interval $(n,2n]$ consists exactly of numbers of the same order of magnitude as $n$, so every edge joins two vertices of comparable size. As $n$ grows, this interval widens proportionally and the number of vertices reachable from $n$ increases accordingly, so that the adjacency patterns of the vertices are naturally distinguished by their labels. As one might expect, the upper bound could in fact be taken to be $cn$; we take $2n$ merely for convenience.

We have $|A|=\aleph_0$ and $|B|\le\aleph_0$ (in fact $|B|=\aleph_0$, although the weaker bound suffices), so $|A|,|B|<|G|$. Moreover, the vertex degrees of the graph $\Gamma$ are
\[
\deg(n)=\#\{m>n:n<m\le 2n\}+\#\{m<n:m<n\le 2m\}
= n+\left(n-\left\lceil\frac{n}{2}\right\rceil\right) = \left\lfloor \frac{3n}{2} \right\rfloor,
\]
so $\deg(n)$ is strictly increasing in $n$ and all vertices have pairwise distinct degrees; this is precisely the reason for our choice of the range of $m$. Since a graph automorphism must preserve the degree of every vertex, any automorphism maps each vertex to itself, and hence $\Aut(\Gamma) = \{\mathrm{id}\}$.

We now try to prove that no witness set $X$ exists.

\begin{proof}[Proof of Theorem~\ref{thm:sep}]
By construction, $A\cap B=\varnothing$, $|A|=\aleph_0$ and $|B|\le\aleph_0<2^{\aleph_0}=|G|$. Assume that there exists $X\subset G$ witnessing the separation of $A$ and $B$. Since $X$ is infinite, there exists $g\in X \setminus \{e\}$. Since $\Aut(\Gamma)=\{\mathrm{id}\}$, $g$ is not an automorphism of $\Gamma$, so there exists an edge $\{i, j\}\in E$ such that $\{g(i), g(j)\}\notin E$. Take $a=(i\;j)\in A$; then
\[
  gag^{-1}=(g(i)\;g(j))\in B.
\]
Since $g\in X$ and $g^{-1}\in X$, we have $gag^{-1}\in XAX$, so $XAX\cap B\neq\emptyset$, a contradiction. Hence no witness set $X$ exists, i.e., $A$ and $B$ are not separated in $G$.
\end{proof}

Thus Question~17.102 is settled, and the answer is negative. One cannot help asking, however: what is the core of this counterexample? Do all counterexamples look like this? Can it lead to a necessary and sufficient condition for separability? In what follows we investigate these questions and study separability in infinite groups in depth.

The key to the counterexample above is the rigidity $\Aut(\Gamma)=\{\mathrm{id}\}$, i.e., the fact that $g$ is not an automorphism of $\Gamma$. This prevents $g$ from preserving the pair $(A,B)$, and thereby excludes the existence of $X$. Once this rigid-graph construction is available, it is easy to produce further counterexamples of the same type. The following two examples were supplied by GPT-5.6 Luna.

\begin{example}\label{ex:hanging-paths}
Let the vertex set be
\[
V=\{v_n:n\ge 0\}\cup\{w_{n,k}:n\ge 0,\ 1\le k\le 2n+10\},
\]
and define a graph $\Gamma=(V,E)$ with edge set
\[
E=\{v_nv_{n+1}:n\ge 0\}\cup\{v_nw_{n,1}:n\ge 0\}\cup\{w_{n,k}w_{n,k+1}:n\ge 0,\ 1\le k\le 2n+9\}.
\]
In other words, $\{v_n\}$ forms a ray, and at each $v_n$ hangs a path of length $2n+10$, namely $v_n-w_{n,1}-\cdots-w_{n,2n+10}$. Let $G=\Sym(V)$ and
\[
A=\{(u\;v)\in G:\{u,v\}\in E\},\qquad
B=\{(u\;v)\in G:\{u,v\}\notin E\}.
\]
Then $A\cap B=\varnothing$ and $|A|=|B|=\aleph_0<2^{\aleph_0}=|G|$, but $A$ and $B$ are not separated in $G$.
\end{example}

\begin{proof}
We first prove that $\Aut(\Gamma)=\{\mathrm{id}\}$. The graph $\Gamma$ is obtained by hanging a finite path at each vertex of the ray $v_0v_1v_2\cdots$; hence $\Gamma$ is a tree, and any two vertices are joined by a unique path. Its leaves are exactly the vertices $\ell_n:=w_{n,2n+10}$ ($n\ge 0$), because all remaining $w_{n,k}$ ($1\le k\le 2n+9$) have degree $2$, while $v_0$ is adjacent only to $v_1$ and $w_{0,1}$ and hence has degree $2$, and $v_n$ ($n\ge 1$) is adjacent to $v_{n-1},v_{n+1},w_{n,1}$ and hence has degree $3$.

In a tree, the distance from a leaf to the nearest vertex of degree $3$ is invariant under automorphisms. For $n\ge 1$, walking inward from $\ell_n$ along the hanging path, all the $w_{n,k}$ passed along the way have degree $2$, and after $2n+10$ edges one first meets a vertex of degree $3$, namely $v_n$; so this distance is $2n+10$ (the values being $12,14,16,\ldots$). Moreover, this nearest vertex is unique, since all other vertices of degree $3$ lie on the ray beyond $v_n$ and are strictly farther. For $n=0$, the walk passes through $v_0$, which has degree $2$, and one further edge is needed to reach the first vertex of degree $3$, namely $v_1$; the distance is $11$, and again $v_1$ is the unique nearest vertex. The numbers $11,12,14,16,\ldots$ are pairwise distinct, so each $\ell_n$ is the unique leaf at its corresponding distance; an automorphism preserves the set of leaves together with these distances, and therefore fixes all $\ell_n$ pointwise.

Fix $n\ge 1$. Since $\ell_n$ is fixed and $v_n$ is the first vertex of degree $3$ on the unique path starting at $\ell_n$ — a position determined intrinsically by the graph structure — $v_n$ is fixed as well. The path from $v_n$ to $\ell_n$ passes successively through $w_{n,1},w_{n,2},\ldots,w_{n,2n+10}=\ell_n$; both endpoints are fixed, so the unique path between them in the tree is preserved as a whole, and each intermediate point is uniquely determined by its distances to the two endpoints, so the vertices $w_{n,k}$ are fixed pointwise. Hence all $v_n$ ($n\ge 1$) and all $w_{n,k}$ ($n\ge 1$) are fixed. Finally, the unique path from $\ell_0$ to $v_1$ passes successively through $w_{0,10},w_{0,9},\ldots,w_{0,1},v_0$ and ends at $v_1$; both endpoints are fixed, and by the same argument $v_0$ and all $w_{0,k}$ are fixed pointwise. All vertices are now fixed, so $\Aut(\Gamma)=\{\mathrm{id}\}$.

Assume that there exists an infinite witness set $X\subset G$. Take $g\in X\setminus\{e\}$. Since $\Aut(\Gamma)=\{\mathrm{id}\}$, $g$ is not an automorphism of $\Gamma$, so there exists an edge $\{i,j\}\in E(\Gamma)$ such that $\{g(i),g(j)\}\notin E(\Gamma)$. Take $a=(i\;j)\in A$; then
\[
gag^{-1}=(g(i)\;g(j))\in B.
\]
Since $g\in X$ and $g^{-1}\in X$, we get $gag^{-1}\in XAX$, whence $XAX\cap B\ne\varnothing$, a contradiction. Therefore $A$ and $B$ are not separated.
\end{proof}

\begin{example}\label{ex:ray}
Take $G=\Sym(\mathbb{N})$ and let
\[
A=\{(n\;n+1):n\ge 0\},\qquad
B=\{(i\;j):0\le i<j,\ j\ne i+1\}.
\]
Then $A\cap B=\varnothing$ and $|A|=|B|=\aleph_0<2^{\aleph_0}=|G|$, but $A$ and $B$ are not separated in $G$.
\end{example}

\begin{proof}
Assume that there exists an infinite witness set $X\subset G$, that is, $e\in X$, $X=X^{-1}$, and $XAX\cap B=\varnothing$. Take any $x\in X$. Since $x^{-1}\in X$, the condition $XAX\cap B=\varnothing$ gives
\[
xAx^{-1}\cap B=\varnothing.
\]
Conjugation sends transpositions to transpositions, and $A\cup B$ is exactly the set of all transpositions in $G$, so $xAx^{-1}\subset A$. Similarly, $x^{-1}\in X$ gives $x^{-1}Ax\subset A$, and conjugating both sides by $x$ yields $A\subset xAx^{-1}$. Therefore
\[
xAx^{-1}=A.
\]
Observe that $A$ is exactly the edge set of the ray graph
\[
0-1-2-3-\cdots
\]
so the equality $xAx^{-1}=A$ means that $x$ is an automorphism of this ray graph. We now show that this ray graph has no automorphism other than the identity.

Let \(\Gamma\) be the ray graph with vertex set \(\mathbb N\) and edge set \(A\). Then
\[
N(0)=\{1\},\qquad N(n)=\{n-1,n+1\}, \quad n\ge 1.
\]
Hence \(\deg(0)=1\), while \(\deg(n)=2\) for \(n\ge 1\), so \(0\) is the unique vertex of degree \(1\) in \(\Gamma\). Let \(\sigma\in\operatorname{Aut}(\Gamma)\). Since graph automorphisms preserve vertex degrees, we must have \(\sigma(0)=0\). Moreover, graph automorphisms preserve adjacency, so
\[
\sigma(1)\in \sigma(N(0))=N(\sigma(0))=N(0)=\{1\},
\]
whence \(\sigma(1)=1\).

We now prove by induction that \(\sigma(n)=n\) for all \(n\ge 0\). The cases \(n=0,1\) have been established. Assume that for some \(n\ge 1\),
\[
\sigma(k)=k, \qquad 0\le k\le n.
\]
Since \(n+1\) is adjacent to \(n\), the vertex \(\sigma(n+1)\) is adjacent to \(\sigma(n)=n\), that is,
\[
\sigma(n+1)\in N(\sigma(n))=N(n)=\{n-1,n+1\}.
\]
But by the induction hypothesis \(\sigma(n-1)=n-1\), and \(\sigma\) is injective, so \(\sigma(n+1)\neq n-1\). Therefore
\[
\sigma(n+1)=n+1.
\]
By induction, \(\sigma(n)=n\) for all \(n\ge 0\), i.e., \(\sigma=\mathrm{id}_{\mathbb N}\). Hence $x=e$.

By the arbitrariness of $x\in X$, we conclude $X=\{e\}$, contradicting the infinitude of $X$. Hence $A$ and $B$ are not separated in $G$.
\end{proof}

All these counterexamples have a rigid graph at their core, but the graph itself is not the essential point: it is merely a simple instance of a binary relation of our choice. We therefore extend these counterexamples to an entire family.

\begin{theorem}\label{thm:l1}
There exists an infinite set $\Omega$ together with a binary relation $\mathcal{R}\subset\Omega\times\Omega$ such that
\begin{enumerate}
      \item $|\mathcal{R}|, |\mathcal{R}^c| < |\Sym(\Omega)|$; if the Axiom of Choice is admitted, this condition may be ignored;
      \item $\Aut(\Omega, \mathcal{R}) := \{\sigma\in\Sym(\Omega): u\mathcal{R}v\iff\sigma(u)\mathcal{R}\sigma(v), \forall u,v \in \Omega\} = \{\mathrm{id}\}$.
\end{enumerate}
\end{theorem}

\begin{proof}
Take $\Omega=\mathbb{N}$ and consider the undirected simple graph $\Gamma=(V, E)$ with vertex set $V = \mathbb{N}$ and edge set $E = \bigl\{\{n,m\} : 0 < n < m \le 2n\bigr\}$. Consider the adjacency relation $\mathcal{R}\subset\mathbb{N}\times\mathbb{N}$ defined by
\[
n\mathcal{R}m\iff\{n,m\} \in E \ \text{and}\ n \ne m.
\]
It is readily seen that
\[
|\mathcal{R}| = |E| = \aleph_0 < 2^{\aleph_0} = |\Sym(\mathbb{N})|.
\]
Since $\mathcal{R}^c\subset\mathbb{N}\times\mathbb{N}$,
\[
|\mathcal{R}^c| \le |\mathbb{N}\times\mathbb{N}| = \aleph_0 < 2^{\aleph_0} = |\Sym(\mathbb{N})|.
\]
An automorphism $\sigma$ satisfies $u\mathcal{R}v\iff\sigma(u)\mathcal{R}\sigma(v)$, which is equivalent to saying that $\sigma$ preserves the adjacency relation, i.e., the edge set $E$. Thus $\Aut(\Omega,\mathcal{R})$ is the automorphism group of the graph $\Gamma$. For any vertex $n$,
\[
\deg(n)=\#\{m>n:n<m\le 2n\}+\#\{m<n:m<n\le 2m\}
= n+\left(n-\left\lceil\frac{n}{2}\right\rceil\right) = \left\lfloor \frac{3n}{2} \right\rfloor,
\]
so $\deg(n)$ is strictly increasing in $n$ and all vertices have pairwise distinct degrees. Since graph automorphisms must preserve the degree of every vertex, any automorphism maps each vertex to itself. Hence $\Aut(\Gamma) = \{\mathrm{id}\}$, i.e., $\Aut(\Omega,\mathcal{R}) = \{\mathrm{id}\}$. Therefore $\Omega=\mathbb{N}$ satisfies the requirements, and the theorem holds.
\end{proof}

\begin{theorem}\label{thm:l2}
Theorem~\ref{thm:l1} $\Rightarrow$ Theorem~\ref{thm:sep}.
\end{theorem}

\begin{proof}
Let $G=\Sym(\Omega)$, $A=\{(u\;v)\in G:u\mathcal{R}v\}$ and $B=\{(u\;v)\in G:u\mathcal{R}^c v\}$, where $(u\;v)$ denotes the transposition swapping $u$ and $v$. Then $A\cap B=\varnothing$, and moreover $|A|=|\mathcal{R}|<|\Sym(\Omega)|=|G|$ and $|B|=|\mathcal{R}^c|<|\Sym(\Omega)|=|G|$. Assume that a witness set $X\subset G$ exists. Take $g\in X\setminus \{e\}$; since $\Aut(\Omega,\mathcal{R})=\{\mathrm{id}\}$, there must exist a transposition $(u\;v)$ such that $u\mathcal{R}v$ and $g(u)\mathcal{R}^cg(v)$. Let $a=(u\;v)\in A$; then $gag^{-1}=(g(u)\;g(v))\in B$. Since $g\in X$ and $g^{-1}\in X^{-1}=X$, we have $gag^{-1}\in XAX$, so $XAX\cap B\ne \varnothing$, a contradiction. Hence no witness set $X$ exists, and Theorem~\ref{thm:sep} holds.
\end{proof}

\section{The Second Family of Counterexamples: Torsion-Free Groups with Countable Square Set}\label{sec:square}

We now ask whether every counterexample is given by a rigid binary relation. Intuition suggests not, so we construct a family of a different kind. In constructing counterexamples one frequently makes use of special groups, such as Tarski monster groups; here we use a group of Newelski.

\begin{theorem}\label{thm:main}
Theorem~\ref{thm:sep} holds.
\end{theorem}

\begin{proof}
Take the group $H$ from Corollary~\ref{thm:newelski}, i.e., $|H| = 2^{\aleph_0}$, $|H^2| = \aleph_0$, $H[2] = \{e\}$. Let
\[
G = H \times C_2, \qquad C_2 = \langle c \rangle,
\]
where $c$ is an element of order $2$, so that $C_2$ has only the two elements $e$ and $c$. Put $a = (e, c) \in G$. For any $(h,\varepsilon)\in G$, since $C_2$ is Abelian,
\[
a(h,\varepsilon)=(e,c)(h,\varepsilon)=(h,c\varepsilon)=(h,\varepsilon c)=(h,\varepsilon)(e,c)=(h,\varepsilon)a;
\]
moreover, $a^2=(e^2,c^2)=(e,e)=e_G$ while $a\neq e_G$. Thus $a$ is an element of order $2$ in the center $\Zent(G)$. Consider the sets
\[
A = \{a\}\subset G, \qquad
B = \{ (s, c) : s \in H^2,\ s \ne e \}\subset G.
\]
Then $|G| = |H| = 2^{\aleph_0}$, so $G$ is infinite. By the definition of $B$, we have $A \cap B = \varnothing$. Furthermore, $|A| = 1 < |G|$ and $|B| = |H^2 \setminus \{e\}| = |H^2| = \aleph_0 < 2^{\aleph_0} = |G|$.

Assume that a witness set $X\subset G$ exists. Since $A = \{a\}$, we have $XAX = XaX$. Take $g = (h, \varepsilon) \in X$. Since $a\in\Zent(G)$,
\[
gag = g^2a = (h^2, e)(e, c) = (h^2, c).
\]
If $h \ne e$, then $H[2]=\{e\}$ gives $h^2 \ne e$, and hence
\[
gag = (h^2, c) \in B,
\]
contradicting $XAX \cap B = \varnothing$. Therefore every element $g = (h, \varepsilon)$ of $X$ satisfies $h = e$, i.e.,
\[
X \subset \{ (e, e), (e, c) \} = \{ e, a \},
\]
so that $|X| \le 2$, a contradiction. Hence no witness set $X$ exists, i.e., $A$ and $B$ are not separated in $G$.
\end{proof}

This counterexample is completely different from the rigid binary relation construction. We again abstract its essence into a family of counterexamples and thereby identify the underlying reason for non-separability.

\begin{corollary}\label{cor:l3}
There exists an infinite group $G$ with $|G^2| < |G|$ and $G[2]$ finite.
\end{corollary}

\begin{proof}
The group $H$ of Corollary~\ref{thm:newelski} satisfies $|H^2| = \aleph_0 < 2^{\aleph_0} = |H|$ and has $H[2] = \{e\}$ finite.
\end{proof}

\begin{theorem}\label{thm:converse}
Corollary~\ref{cor:l3} $\Rightarrow$ Theorem~\ref{thm:sep}.
\end{theorem}

\begin{proof}
Take the group $H$ provided by Corollary~\ref{cor:l3}, and let $G = H \times C_2$, $C_2 = \langle c \rangle$ with $c^2 = e$, $a = (e, c)$, $A = \{a\}\subset G$ and $B = \{ (s, c) : s \in H^2,\ s \ne e \}\subset G$. First, $A \cap B = \varnothing$ is clear. Also $|A| = 1 < |G|$ and
\[
|B| = |\{ (s, c) : s \in H^2,\ s \ne e \}| = |\left(H^2 \setminus \{e\}\right)\times\{c\}| \le |H^2| < |H| = |G|.
\]
Suppose that a witness set $X\subset G$ exists. Take $g = (h, \varepsilon) \in X$; then $gag = (h^2, c)$. Since $XAX \cap B = \varnothing$, we have $gag \notin B$, so $h^2 = e$, i.e., $h \in H[2]$. Hence $X \subset
H[2] \times C_2$, and therefore $|X| \le |H[2] \times C_2| < \aleph_0$. Thus no witness set $X$ exists, and Theorem~\ref{thm:sep} holds.
\end{proof}

\section{A Necessary and a Sufficient Condition for Separability}\label{sec:W}

Theorem~\ref{thm:l2} shows that Theorem~\ref{thm:l1} yields one family of counterexamples, and Theorem~\ref{thm:converse} shows that Corollary~\ref{cor:l3} yields another. The two families are essentially different, yet there ought to be a common mechanism governing separability. Although we have successfully generalized several counterexamples into two families, the underlying essence has not been touched. We now consider the most crucial ingredient of separability --- the witness set.

Realizing that a direct attack on the witness set $X$ affords no handle, we consider instead a set $W$ that contains all witness sets $X$; this will be the main tool in the arguments below, and it naturally opens up deeper questions.

\begin{definition}\label{def:w}
The set
\[
W=W_G(A,B):=\bigl\{\,g\in G \;:\; \{g,g^{-1}\}\,A\,\{g,g^{-1}\}\cap B=\varnothing\,\bigr\}
\]
is called the \emph{$W$-witness set} of $A$ and $B$, for short the \emph{$W$-set}.
\end{definition}

Clearly, $e\in W$, $W=W^{-1}$, and $g\in W$ if and only if the four sets $gAg,\ gAg^{-1},\ g^{-1}Ag,\ g^{-1}Ag^{-1}$ are all disjoint from $B$. Note that $W$ need not itself be a witness set.

This immediately gives a necessary condition for the separability of $A$ and $B$.

\begin{lemma}[Necessary condition]\label{lem:necessary}
If $X$ satisfies (3) and (4) of Definition~\ref{def:sep}, then $X\subset W$. In particular, if $W$ is finite, then $A$ and $B$ cannot be separated.
\end{lemma}

\begin{proof}
Let $x\in X$. By (3) we have $x^{-1}\in X$, so $\{x,x^{-1}\}\subset X$, whence
\[
\{x,x^{-1}\}A\{x,x^{-1}\}\subset XAX.
\]
By (4), this set is disjoint from $B$, i.e., $x\in W$. Thus $X\subset W$. If $W$ is finite, then every set $X$ satisfying (3)(4) is finite and cannot satisfy (1).
\end{proof}

This indicates one way of refuting Question~17.102: if one can find an example in which $W$ is finite, the question fails. In the opposite direction, $W$ also yields a sufficient condition for separability.

\begin{lemma}[Sufficient condition]\label{lem:greedy}
Let $A\cap B=\varnothing$ and $\lambda:=\max\{|A|,|B|,\aleph_0\}$.
If $|W|>\lambda$, then $A$ and $B$ are separated; moreover, a countable witness set $X$ can be chosen.
\end{lemma}

\begin{proof}
We construct an increasing chain of finite symmetric sets $\{e\}=X_0\subset X_1\subset\cdots$ such that for each $k$ we have $X_k\subset W$ and $X_kAX_k\cap B=\varnothing$.
Proceed by induction on $k$; clearly $X_0=\{e\}$ meets the requirements. Suppose that $X_k$ has been constructed. We choose $x\in W$ and set $X_{k+1}=X_k\cup\{x,x^{-1}\}$. Three kinds of $x$ must be excluded:
\begin{enumerate}
\item[(i)] To ensure that the cross terms between old and new elements avoid $B$. For $y\in X_k$, $a\in A$, $b\in B$, the equations $xay=b$, $yax=b$, $x^{-1}ay=b$, $yax^{-1}=b$ uniquely determine, respectively,
\[
x=b\,y^{-1}a^{-1},\quad x=a^{-1}y^{-1}b,\quad x=\bigl(b\,y^{-1}a^{-1}\bigr)^{-1},\quad
x=\bigl(a^{-1}y^{-1}b\bigr)^{-1}.
\]
Hence there are at most $4\,|X_k|\,|A|\,|B|\le\lambda$ such forbidden $x$; here the Axiom of Choice is used.

\item[(ii)] To ensure that the terms among the new elements avoid $B$: since $x\in W$, this holds automatically.

\item[(iii)] To ensure a strict enlargement: exclude the elements of the finite set $X_k\cup X_k^{-1}$.
\end{enumerate}

The total number of forbidden $x$ is $\le\lambda<|W|$, so a legitimate $x\in W$ exists. Let $X=\bigcup_{k=0}^{\infty}X_k$. Then $e\in X$, $X=X^{-1}$, $|X|=\aleph_0$; moreover, for any $x,x'\in X$ and $a\in A$, there is $k$ with $x,x'\in X_k$, and by construction $xax'\notin B$. Hence $X$ is a witness set.
\end{proof}

Thus $W$ is a device that localizes the problem: it yields both a necessary and a sufficient condition for separability and indicates the direction in which counterexamples should be sought.

Taken together, Lemma~\ref{lem:necessary} and Lemma~\ref{lem:greedy} show that in order to construct a counterexample, one must make $W$ very small, at least $|W|\le\lambda$. The condition ``$|A|<|G|$ and $|B|<|G|$'' imposes no constraint on $W$ whatsoever, because $W$ is determined by the conjugation structure of $A$ relative to $G$, a structure independent of the cardinalities $|A|,|B|$.

It is then natural to ask whether the two families of counterexamples generalized above both satisfy the necessary condition on $W$, that is, whether their $W$ is small enough. We now compute the $W$-witness sets in both constructions.

\begin{proposition}\label{prop:W-l2}
In the construction of Theorem~\ref{thm:l2}, where $G=\Sym(\Omega)$ and $A,B$ are the sets of transpositions corresponding to the relation $\mathcal R$ and its complementary relation $\mathcal R^c$, respectively, one has
\[
W_G(A,B)=\{\mathrm{id}\}.
\]
\end{proposition}

\begin{proof}
Let $g\in W_G(A,B)$. By the definition of the $W$-witness set,
\[
gAg^{-1}\cap B=\varnothing,\qquad g^{-1}Ag\cap B=\varnothing.
\]
Since $g(u\;v)g^{-1}=(g(u)\;g(v))$, the set $gAg^{-1}$ again consists of transpositions. As $A\sqcup B$ is exactly the set of all transpositions in $G$, it follows that $gAg^{-1}\subset A$; similarly $g^{-1}Ag\subset A$. Then $A\subset gAg^{-1}$, so $gAg^{-1}=A$, and likewise $g^{-1}Ag=A$. By the definition of $A$,
\[
u\mathcal R v\iff (u\;v)\in A\iff (g(u)\;g(v))\in A\iff g(u)\mathcal R g(v).
\]
Hence $g\in\Aut(\Omega,\mathcal R)=\{\mathrm{id}\}$, i.e., $g=\mathrm{id}$. Conversely, since $A\cap B=\varnothing$, clearly $\mathrm{id}\in W_G(A,B)$. Therefore $W_G(A,B)=\{\mathrm{id}\}$.
\end{proof}

\begin{proposition}\label{prop:W-converse}
In the construction of Theorem~\ref{thm:converse}, where $G=H\times C_2$, $a=(e,c)$, $A=\{a\}$ and $B=\{(s,c):s\in H^2,\ s\ne e\}$, one has
\[
W_G(A,B)=H[2]\times C_2.
\]
\end{proposition}

\begin{proof}
Take $g=(h,\varepsilon)\in G$. For any $(h',\varepsilon')\in G=H\times C_2$,
\[
a(h',\varepsilon')=(e,c)(h',\varepsilon')=(h',c\varepsilon')=(h',\varepsilon'c)=(h',\varepsilon')(e,c)=(h',\varepsilon')a,
\]
so $a=(e,c)\in\Zent(G)$. Thus the four products take only three values:
\[
gag=g^2a=(h^2,c),\qquad gag^{-1}=g^{-1}ag=a=(e,c),\qquad
g^{-1}ag^{-1}=g^{-2}a=(h^{-2},c).
\]
Here $(e,c)\notin B$ is clear; furthermore, since $h^2\in H^2$ and $h^{-2}=(h^{-1})^2\in H^2$,
\[
(h^2,c),(h^{-2},c)\notin B\iff h^2=e\iff h\in H[2],
\]
while the second component $\varepsilon\in C_2$ is arbitrary. Therefore $g\in W_G(A,B)$ if and only if $h\in H[2]$, i.e., $W_G(A,B)=H[2]\times C_2$.
\end{proof}

It follows that the $W$-witness sets of both families of counterexamples are finite, so both satisfy the necessary condition of Lemma~\ref{lem:necessary}, which is precisely the intrinsic reason why they cannot be separated. The first family compresses $W$ into a finite set by means of a rigid binary relation, and the second by means of the finiteness of $H[2]$. Thus the heart of the two families found earlier consists of two distinct mechanisms for compressing $W$ to a finite set; these are not the direct cause of non-separability itself.

By Lemmas \ref{lem:necessary} and \ref{lem:greedy}, separability is almost completely determined by the size of $W$, while the size of $W$ depends on the specific conjugation relation of $A$ relative to $G$ and neither implies nor is implied by the cardinalities $|A|, |B|$. Where Question 17.102 goes wrong is precisely in using cardinality conditions to constrain conjugation conditions.

With the necessary condition and the sufficient condition in hand, we can state a separation theorem.

\begin{theorem}\label{thm:main-negative}
Let $G$ be an infinite group and let $A,B$ be disjoint subsets of $G$, and let $\lambda=\max\{|A|,|B|,\aleph_0\}$. If $|W_G(A,B)|>\lambda$, then $A$ and $B$ are separated; conversely, if $W_G(A,B)$ is finite, then they are not separated.
\end{theorem}

\section{$\aleph_0\le|W|\le\lambda$}\label{sec:alt}

Naturally, we would like to know what happens to separability when $\aleph_0\le|W|\le\lambda$. If one wants to refute Question 17.102, must $W$ necessarily be finite? Using the Tencent Hy3 model, guided by the preceding material, we obtained the counterexamples and proofs of Sections \ref{sec:S3} and \ref{sec:VxU} below.
This section presents two counterexamples showing that even when $|W|=\lambda\;(=\aleph_0)$, the pair $A,B$ may still fail to be separated. This shows that the strict inequality $|W| > \lambda$ in Lemma~\ref{lem:greedy} cannot be weakened to $|W| \ge \lambda$ or to $|W| \ge \aleph_0$.

\subsection{The group of $S_3$-valued functions}\label{sec:S3}

Let $S_3$ be the symmetric group on three letters and $A_3$ its alternating subgroup.
Take
\[
G = \bigl\{ f \in \prod_{n\in\mathbb{N}} S_3 : f(n) \in A_3 \text{ for all but finitely many } n \bigr\},
\]
with the operation carried out coordinatewise. Write $\tau = (12) \in S_3$ and define
\[
a_n \in G, \qquad a_n(m) = \begin{cases}
\tau, & m = n, \\
e,   & m \ne n,
\end{cases}
\]
as well as $E = \bigl\{ f \in G : f(n) \in \{e,\tau\}, \forall n \in \mathbb{N} \bigr\}$. Take $A = \{ a_n : n \in \mathbb{N} \}$ and
\[
B = \bigl\{ f\in G: \exists n\in\mathbb{N}\;\bigl( f(n)\in\{(13),(23)\} \land \forall m\ne n,\; f(m)=e \bigr) \bigr\}\cup\bigl( E \setminus (A \cup \{e\}) \bigr).
\]

\begin{proposition}\label{prop:sep}
For the above $G,A,B$, we have:
\begin{enumerate}
  \item[(1)] $G$ is a group and $|G| = 2^{\aleph_0}$;
  \item[(2)] $A \cap B = \varnothing$, and $|A| = |B| = \aleph_0 < |G|$;
  \item[(3)] $W_G(A,B) = E$, so that $|W| = \aleph_0$;
  \item[(4)] every set $X \subset G$ satisfying (2)(3)(4) of Definition~\ref{def:sep} is necessarily $\{e\}$, and therefore $A$ and $B$ are not separated.
\end{enumerate}
\end{proposition}

\begin{proof}
(1) Since $A_3<S_3$ and the operation is coordinatewise, the condition ``belonging to $A_3$ except for finitely many coordinates'' is closed under coordinatewise multiplication and inversion, so $G$ is a group. From $\prod_{n} A_3 \subset G$ we get $|G|\ge |\prod_{n} A_3| = 3^{\aleph_0} = 2^{\aleph_0}$; on the other hand, for each finite set $F \subset \mathbb{N}$ put $G_F = S_3^F \times A_3^{\mathbb{N}\setminus F}$; then $G = \bigcup_{F} G_F$ and $|G_F| = 6^{|F|} \cdot 3^{\aleph_0} = 2^{\aleph_0}$, so $|G| \le \aleph_0 \cdot 2^{\aleph_0} = 2^{\aleph_0}$. Hence $|G| = 2^{\aleph_0}$.

(2) Since the functions in $E$ take only $e$ or $\tau$ in each coordinate, they must have finite support in order to satisfy the defining condition of $G$; conversely, every function with finite support and values in $\{e,\tau\}$ lies in $E$. Hence $|E| = \aleph_0$. We have $A \subset E$, and clearly $A\cap B=\emptyset$. Moreover, $|A| = \aleph_0$ and $|B| = 2\aleph_0 + (|E| - |A| - 1) = \aleph_0 < |G|$.

(3) By definition, $W = \{ g \in G : \{g,g^{-1}\}A\{g,g^{-1}\} \cap B = \varnothing \}$. Let $g \notin E$. Then there exists $n\in\mathbb{N}$ with $g(n) \notin \{e,\tau\} = C_{S_3}(\tau)$. Take $a = a_n \in A$; the $n$-th coordinate of $g a g^{-1}$ is $g(n) \tau g(n)^{-1}\neq\tau$, hence equals $(13)$ or $(23)$, so $g a g^{-1}$ falls into the first part of $B$, and $g \notin W$.

Conversely, let $g \in E$. Then each coordinate $g(m) \in \{e,\tau\}$, so $g(m)^2 = e$, whence $g^{-1} = g$. For any $a = a_n$, since $g(n) \in \{e,\tau\}$ commutes with $\tau$, we get $g a g = a$. Thus $\{g,g^{-1}\} A \{g,g^{-1}\} = g A g = A$, which is disjoint from $B$, and so $g \in W$. Hence $W = E$ and $|W| = \aleph_0$.

(4) Let $X \subset G$ satisfy $e \in X$, $X = X^{-1}$, $XAX \cap B = \varnothing$. By Lemma~\ref{lem:necessary}, $X \subset W = E$. If there exists $x \in X$ with $x \neq e$, then $x \in E$ and the support $\supp(x)$ is a nonempty finite set. Hence there exists $n \notin \supp(x)$. Then
\[
x a_n = x a_n e \in XAX,
\]
so $x a_n \notin B$. For $m\neq n$ we have $xa_n(m)=x(m)\in \{e, \tau\}$, so $m\in\supp(xa_n)\iff m\in\supp(x)$; for $m=n$ we have $xa_n(n)=\tau\neq e\iff n\in\supp(xa_n)$. Therefore $\supp(xa_n)=\supp(x)\cup\{n\}$ is finite and $xa_n\in E$. Moreover, since $|\supp(x a_n)| \ge 2$, the element $x a_n$ is neither $e$ nor one of the $a_m$ having a single nontrivial coordinate, so $x a_n \in E \setminus (A \cup \{e\}) \subset B$, a contradiction. Hence $X$ has no nontrivial element, and since $e \in X$ we get $X = \{e\}$, a finite set that fails (1) of Definition~\ref{def:sep}. Therefore $A$ and $B$ are not separated.
\end{proof}

\subsection{The semidirect product $\mathbf{V \rtimes U}$}\label{sec:VxU}

Let $V = \bigoplus_{i \in \mathbb{N}} \mathbb{F}_2 e_i$ be the countable-dimensional vector space over $\mathbb{F}_2$ with basis $\{e_i\}_{i\in\mathbb{N}}$.
Define
\[
U = \Bigl\{ u \in \operatorname{GL}(V) : u(e_i) = e_i + \sum_{j < i} c_{ij} e_j,\; c_{ij} \in \mathbb{F}_2 \Bigr\},
\]
that is, the group of automorphisms of $V$ represented, with respect to the basis $\{e_i\}_{i\in\mathbb{N}}$, by unipotent lower triangular matrices. Let $G = V \rtimes U$ with multiplication
\[
(v, u)(w, s) = (v + u(w),\, u s).
\]
Take $A = \{ (e_i, 1) : i \in \mathbb{N} \}$ and $B = \{ (v, 1) : v \in V \setminus (\{0\} \cup \{e_i : i \in \mathbb{N}\}) \}$.

\begin{proposition}
\label{prop:sep2}
For the above $G,A,B$, we have:
\begin{enumerate}
  \item[(1)] $U$ is a group, $|U| = 2^{\aleph_0}$, and $|G| = 2^{\aleph_0}$;
  \item[(2)] $A \cap B = \varnothing$, and $|A| = |B| = \aleph_0 < |G|$;
  \item[(3)] $W_G(A,B) = V \times \{1\}$, so that $|W| = \aleph_0$;
  \item[(4)] every set $X$ satisfying (2)(3)(4) of Definition~\ref{def:sep} is necessarily $\{e\}$, so $A$ and $B$ are not separated.
\end{enumerate}
\end{proposition}

\begin{proof}
(1) After fixing the basis $\{e_i\}$, $U$ consists exactly of the unipotent lower triangular infinite matrices with finitely many nonzero columns. The product of two such matrices is again unipotent lower triangular with finitely many nonzero columns, so $U$ is closed under multiplication. For $u=I+N\in U$, we have $N(e_i)\in\operatorname{span}\{e_j:j<i\}$, whence $N^k(e_i)=0$ for $k\ge i$, i.e., $N$ is locally nilpotent; therefore $u^{-1}=\sum_{k\ge0}(-N)^k$ is a finite sum on each vector and is again unipotent lower triangular, so $u^{-1}\in U$, and $U$ is a group.

$U$ is determined by the family $\{c_{ij}\in \mathbb{F}_2 : j<i, i\in\mathbb{N}\}\cong\{0,1\}^{\{(i,j): j<i\}}$, and hence has cardinality $2^{\aleph_0}$. Moreover, the map $\varphi: V \to \{\text{finite subsets of }\mathbb{N}\}$, $v\mapsto \supp(v)$, is a bijection and
\[
   |\{\text{finite subsets of }\mathbb{N}\}|=\sum_{k\in\mathbb{N}}\binom{\aleph_0}{k}=\aleph_0,
\]
so $|V|=\aleph_0$. Therefore $|G| = |V| \cdot |U| = \aleph_0 \cdot 2^{\aleph_0} = 2^{\aleph_0}$.

(2) Clearly $A\cap B=\emptyset$. We have $|A| = \aleph_0$ and $|B| = |V \setminus (\{0\} \cup \{e_i : i \in \mathbb{N}\})| = \aleph_0$, both smaller than $|G|$.

(3) For $x = (v, u) \in G$ we have $x^{-1} = (u^{-1}(v), u^{-1})$. Hence
\[
x (e_i, 1) x^{-1}
= (v + u(e_i), u) (u^{-1}(v), u^{-1})
= (v + u(e_i) + v,\; 1)
= (u(e_i), 1).
\]

If $u \neq 1$, choose $i$ with $u(e_i) \neq e_i$. Since $u(e_i) = e_i + \sum_{j < i} c_{ij} e_j$ has $e_i$-coefficient $1$, we have $u(e_i) \neq 0$; and if $u(e_i) = e_k$, comparing the $e_i$-coefficients gives $k = i$, a contradiction. Hence $u(e_i) \in V \setminus (\{0\} \cup \{e_j\})$, i.e., $x(e_i,1)x^{-1} \in B$, so $x \notin W$.

If $u = 1$, then $x = (v, 1)$ and $x^{-1} = (v, 1) = x$. In this case $x(e_i, 1)x = (v + e_i, 1)(v, 1) = (v + e_i + v, 1) = (e_i, 1) \in A$, so $x(e_i,1)x\notin B$, and therefore $\{x\}A\{x\}\cap B\subset A\cap B=\emptyset$, i.e., $x \in W$. In summary, $W = V \times \{1\}$, and $|W| = \aleph_0$.

(4) By Lemma~\ref{lem:necessary}, $X \subset W = V \times \{1\}$. Suppose there exists $x = (v, 1) \in X$ with $v \neq 0$. Since $e \in X$, taking $x' = e$ gives
\[
x (e_i, 1) = (v + e_i, 1) \in XAX.
\]
The support $\supp(v)$ of $v$ is a nonempty finite set; choose $i \notin \supp(v)$. For $j\neq i$ we have $j\in \supp(v+e_i)\iff j\in\supp(v)$; for $j=i$, since $i\notin\supp(v)$, the $e_i$-coefficient $c_i$ of $v$ is $0$, so $1+c_i=1\neq 0$, and hence $i\in\supp(v+e_i)$. In summary, $\supp(v+e_i)=\supp(v)\cup\{i\}$. In particular, $\supp(0)=\emptyset$ and $\supp(e_j)=\{j\}$. Hence $|\supp(v+e_i)|=|\supp(v)|+1\ge 2$, so $v+e_i\notin\{0\} \cup \{e_j: j\in\mathbb{N}\}$, i.e., $(v+e_i, 1)\in B$, a contradiction. Therefore $X$ has no element with $v \neq 0$, and since $e \in X$ we get $X = \{e\}$, a finite set. Hence $A$ and $B$ are not separated.
\end{proof}

The above two counterexamples show that $W$ can be infinite, and that non-separability may still arise because $X$ is forced into a very small part of $W$. Hence the strict inequality $|W| > \lambda$ in Lemma~\ref{lem:greedy} is a sufficient condition for separability but not a necessary one; when $\aleph_0 \le |W| \le \lambda$, whether the pair is separable depends on finer structure.

We notice that the two AI-generated examples above look completely different in form, yet their proofs are entirely parallel, which suggests that a common mechanism is behind their non-separability. Let us first examine this mechanism; it might reveal a clue toward a separability criterion for the range $\aleph_0 \le |W| \le \lambda$.

Comparing the two proofs, the most conspicuous feature is the similar role played by $A$: in Proposition~\ref{prop:sep2}, the set $A=\{(e_i,1):i\in\mathbb{N}\}$ is itself an $\mathbb{F}_2$-basis of $V\times\{1\}$; in Proposition~\ref{prop:sep}, the elements of $E$ take $e$ or $\tau$ in each coordinate, are pairwise commuting and of order dividing $2$, and each $f\in E$ is uniquely determined by its finite support, being exactly the product of the $a_n$ over its support; hence $A=\{a_n:n\in\mathbb{N}\}$ is likewise an $\mathbb{F}_2$-basis of $E$. This suggests a uniform notation: let $V=\bigoplus_{a\in A}\mathbb{F}_2 e_a$; then $e_a\mapsto a$ gives, in both examples, a group embedding $\iota\colon(V,+)\hookrightarrow(G,\cdot)$, with image $V\times\{1\}$ and $E$, respectively.

The overlap goes further. In both examples we have $\iota(V)\setminus(\{e\}\cup A)\subset B$: in Proposition~\ref{prop:sep2} this is directly part of the definition of $B$, and in Proposition~\ref{prop:sep} the set $E\setminus(A\cup\{e\})$ is part of $B$. When computing $W$, in both examples an element outside the embedded subgroup is excluded by exhibiting $a\in A$ with $gag^{-1}\in B$: Proposition~\ref{prop:sep} takes the coordinate $n$ with $g(n)\notin\{e,\tau\}$, and Proposition~\ref{prop:sep2} takes $i$ with $u(e_i)\neq e_i$; elements inside the subgroup commute pairwise, act trivially by conjugation, and belong to $W$ by $A\cap B=\varnothing$. Thus both examples give $W=\iota(V)$. Finally, by Lemma~\ref{lem:necessary}, $X\subset W=\iota(V)$; for a nontrivial $x=\iota(v)\in X$, choose $a_0\in A$ avoiding the support of $v$; then $xa_0\in XAX$ (as $e\in X$), its support has at least two elements, and by the inclusion above it falls into $B$, contradicting $XAX\cap B=\varnothing$.

In retrospect, the two proofs use no special property of $S_3$ or of the semidirect product, but only the embedding $\iota$, the inclusion above, the method of excluding elements by conjugation, and a support count. This suggests that the two counterexamples are two realizations of one general phenomenon; writing these common properties as assumptions yields the following proposition.

\begin{proposition}\label{prop:skeleton}
Let $G$ be an infinite group and let $A\subset G$ be an infinite subset. Let $V=\bigoplus_{a\in A}\mathbb{F}_2 e_a$ be the $\mathbb{F}_2$-vector space with basis $\{e_a:a\in A\}$, and suppose there is a group embedding $\iota\colon(V,+)\hookrightarrow(G,\cdot)$, $e_a\mapsto a$. Suppose further that $B\subset G$ satisfies $A\cap B=\varnothing$ and
\begin{enumerate}
  \item[(i)] $\iota(V)\setminus(\{e\}\cup A)\subset B$;
  \item[(ii)] $\forall g\in G\setminus\iota(V)\ \exists a\in A\ \bigl(gag^{-1}\in B \lor g^{-1}ag\in B\bigr)$.
\end{enumerate}
Then:
\begin{enumerate}
  \item $W_G(A,B)=\iota(V)$;
  \item every set $X\subset G$ satisfying (2)(3)(4) of Definition~\ref{def:sep} is necessarily $\{e\}$,
\end{enumerate}
so that $A$ and $B$ cannot be separated.
\end{proposition}

\begin{remark}
In this situation $|\iota(V)|=|V|=|A|$; if moreover $|B|\le |A|$ and $|G|>|A|$, then with $\lambda=\max\{|A|,|B|,\aleph_0\}=|A|$ we obtain a boundary example of non-separability with $|W|=\lambda$.
\end{remark}

\begin{proof}
We first prove $W=\iota(V)$. If $g\in\iota(V)$, say $g=\iota(v)$, then, since $V$ has characteristic $2$, we have $g^{-1}=\iota(-v)=\iota(v)=g$; and since $V$ is Abelian, for any $a\in A$,
\[
gag^{-1}=\iota(v+e_a-v)=\iota(e_a)=a.
\]
Hence $\{g\}A\{g\}=A$, and by $A\cap B=\varnothing$ we get $g\in W$. Conversely, if $g\notin\iota(V)$, condition (ii) provides $a\in A$ with $gag^{-1}\in B$ or $g^{-1}ag\in B$, so $\{g,g^{-1}\}A\{g,g^{-1}\}\cap B\ne\varnothing$, i.e., $g\notin W$. Therefore $W=\iota(V)$.

We now show that every set satisfying (2)(3)(4) of Definition~\ref{def:sep} is necessarily $\{e\}$. Let $X$ satisfy (2)(3)(4); by Lemma~\ref{lem:necessary}, $X\subset W=\iota(V)$. Every $v\in V$ is uniquely a finite sum $v=\sum_{a\in S}e_a$ with $S\subset A$ finite; call $S$ the support of $v$, denoted $\supp(v)$. Suppose, for contradiction, that $x=\iota(v)\in X\setminus\{e\}$; then $v\ne0$ and $S=\supp(v)$ is nonempty and finite. Since $A$ is infinite, we may take $a_0\in A\setminus S$. For $a\neq a_0$ we have $a\in\supp(v+e_{a_0})\iff a\in S$; for $a=a_0$, since $a_0\notin S$, the $e_{a_0}$-coefficient of $v$ is $0$ and becomes $1\neq 0$ after adding $1$, so $a_0\in\supp(v+e_{a_0})$. In summary, $\supp(v+e_{a_0})=S\cup\{a_0\}$, and hence $|\supp(v+e_{a_0})|=|S|+1\ge 2$. Therefore $\iota(v+e_{a_0})\in\iota(V)\setminus(\{e\}\cup A)$, and by condition (i), $\iota(v+e_{a_0})\in B$. On the other hand, since $e\in X$,
\[
\iota(v+e_{a_0})=\iota(v)\iota(e_{a_0})=xa_0e\in XAX,
\]
contradicting (4), $XAX\cap B=\varnothing$. Hence $X=\{e\}$, a finite set, and therefore $A$ and $B$ cannot be separated.
\end{proof}

We see that, by means of certain delicate structural features, this family demonstrates that even the critical case $|W|=\lambda$ can be non-separable. This shows that the $W$-witness set cannot decide the case $\aleph_0 \le |W| \le \lambda$, and we must look for a more intrinsic reason for separability.

\section{Some Natural Questions Arising from the Preceding Material}\label{sec:further}

One can see that all the counterexamples constructed above satisfy $|G|=2^{\aleph_0}$, whereas the extended families of counterexamples in Theorem~\ref{thm:l1}, Theorem~\ref{thm:converse}, and Proposition~\ref{prop:skeleton} impose no restriction whatsoever on the cardinality of $G$. It is therefore natural to wonder whether the size of $G$ has any essential effect on separability.

\subsection{$|G|=\aleph_0$}

If $|G|=\aleph_0$, then $|A|,|B|<|G|$ forces $A,B$ to be finite, so $\lambda=\aleph_0=|G|$, and the hypothesis $|W|>\lambda$ of Lemma \ref{lem:greedy} cannot be satisfied; a separate argument is required. This is precisely the case in which Protasov asserted that separability holds. We begin with its proof.

\begin{theorem}\label{thm:finite}
Let $G$ be an infinite group and let $A,B$ be finite disjoint subsets of $G$. Then $A,B$ can be separated.
\end{theorem}

\begin{proof}
In Corollary~\ref{lem:zariski} take
\[
F_0=B^{-1},\qquad F_1=A,\qquad F_2=\{e\}.
\]
Since $A\cap B=\varnothing$, we have $e\notin F_0F_1F_2=B^{-1}A$: otherwise there would exist $b\in B$ and $a\in A$ with $b^{-1}a=e$, i.e., $a=b\in A\cap B$, a contradiction. Hence Corollary~\ref{lem:zariski} provides an infinite symmetric set $X\ni e$ with $e\notin B^{-1}XAX$. Now
\[
e\in B^{-1}XAX\iff \exists\, b\in B,\ x_1,x_2\in X,\ a\in A,\ x_1ax_2=b\iff XAX\cap B\ne\varnothing,
\]
so $XAX\cap B=\varnothing$. Thus $X$ is a witness set for $A$ and $B$.
\end{proof}

We take the opportunity to give the proof of the statement, due to Protasov, that separation always holds when $G$ is Abelian. In fact, this can also be proved via Banakh--Protasov \cite{zariski2010}, Theorem 3, but that would introduce more definitions, so we only give an elementary proof below.

\begin{theorem}\label{thm:abel}
Let $G$ be an infinite Abelian group, and let $A,B$ be disjoint subsets of $G$ with $|A|,|B|<|G|$. Then $A$ and $B$ are separated in $G$.
\end{theorem}

\begin{proof}
Since Theorem~\ref{thm:finite} already settles the case where both $A$ and $B$ are finite, it remains to consider the case in which at least one of $A,B$ is infinite. Let $\kappa=|G|$ and $C=BA^{-1}$. Since $A\cap B=\varnothing$, we have $e\notin C$. Also
\[
|C|\le|A||B|=\max\{|A|,|B|\}\le |G|=\kappa.
\]
It suffices to construct an infinite symmetric set $X=X^{-1}\ni e$ with $XX\cap C=\varnothing$. Indeed, if $x_1ax_2=b$ with $x_1,x_2\in X$, $a\in A$, $b\in B$, then, $G$ being Abelian, $x_1x_2=ba^{-1}\in C$, contradicting $XX\cap C=\varnothing$; hence $XAX\cap B=\varnothing$, and $X$ is a witness set.

We recursively construct finite symmetric sets $\{e\}=X_0\subset X_1\subset\cdots$ with $X_nX_n\cap C=\varnothing$. Suppose $F=X_n$ has been constructed, and put
\[
D=C\cup C^{-1},\qquad K=\{x\in G:x^2=e\},\qquad E=FC\cup C^{-1}F.
\]
Since $|C|<\kappa$ and $F$ is finite, $|D|\le 2|C|<2\kappa=\kappa$ and $|E|\le 2|F||C|=|C|<\kappa$. We choose the new element $x$ according to two cases.

If $|K|=\kappa$, take $x\in K\setminus(E\cup F)$. Then $x^2=e\notin C$ and $x^{-2}=e\notin C$, and moreover $x\notin E$.

If $|K|<\kappa$: for a fixed $d\in G$, if the equation $z^2=d$ has a solution $z_0$, then $\{z\in G: z^2=d\}=Kz_0$. Indeed, on the one hand, if $z^2=z_0^2$, then, since $G$ is Abelian, $(zz_0^{-1})^2=z^2z_0^{-2}=e$, so $zz_0^{-1}\in K$; on the other hand, every $k\in K$ satisfies $(kz_0)^2=k^2z_0^2=d$. Hence the equation has at most $|K|$ solutions. Each $d\in D$ contributes at most $|K|$ solutions, so
\[
|\{z\in G:z^2\in D\}|\le|D||K|<\kappa,
\]
and we may take $x\notin F\cup E\cup\{z\in G:z^2\in D\}$. Then $x^2\notin C\cup C^{-1}$, i.e., $x^2,x^{-2}\notin C$, and still $x\notin E$.

In both cases put $X_{n+1}=X_n\cup\{x,x^{-1}\}$. This set is finite and symmetric, and $x\notin F$ ensures $X_{n+1}\supsetneq X_n$. Now
\[
X_{n+1}X_{n+1}=(X_n\cup\{x,x^{-1}\})(X_n\cup\{x,x^{-1}\})=X_nX_n\;\cup\;X_n\{x,x^{-1}\}\;\cup\;\{x^2,\,e,\,x^{-2}\}:
\]

$X_nX_n\cap C=\emptyset$ by hypothesis. If $fx\in C$ or $fx^{-1}\in C$ for some $f\in X_n$, then respectively $x\in f^{-1}C\subset FC$ and $x\in C^{-1}f\subset C^{-1}F$, both contradicting $x\notin E$. Finally, $e, x^2, x^{-2}\notin C$. Therefore $X_{n+1}X_{n+1}\cap C=\varnothing$.

Let $X=\bigcup_{n\ge0}X_n$. Then $X$ is infinite, $e\in X$, and $X=X^{-1}$. For any $x_1, x_2\in X$, choose $n_1, n_2\in \mathbb{N}$ with $x_1\in X_{n_1}$ and $x_2\in X_{n_2}$, and let $n=\max\{n_1,n_2\}$; since the chain is increasing, $x_1,x_2\in X_n$, so $x_1x_2\notin C$. Hence $XX\cap C=\varnothing$, and therefore $XAX\cap B=\varnothing$, so $X$ is a witness set.
\end{proof}

We have also seen that, quite apart from the range $\aleph_0 \le |W| \le \lambda$, separability is straightforward when $A,B$ are finite; yet Lemma~\ref{lem:greedy} does not give a definitive verdict even there. This shows that the $W$-witness set is still an imperfect tool, and that more powerful tools are needed.

Theorem~\ref{thm:abel} also imposes no restriction on the size of $G$, which suggests that the cardinality of $G$ has essentially no bearing on separability. We thus pose the following question: for an infinite cardinal $\kappa$ that cannot be written in the form $2^{\lambda}$, does there exist a counterexample with $|G|=\kappa$?

\subsection{$|G|=\kappa$ for an infinite cardinal $\kappa$ not of the form $2^{\lambda}$}

\begin{theorem}\label{thm:spectrum}
Let $\kappa$ be an infinite cardinal. If there exists an infinite cardinal $\mu<\kappa$ with $\kappa<\Ded\mu$, then there exists a counterexample of cardinality exactly $\kappa$.
\end{theorem}

\begin{proof}
By Lemma~\ref{lem:newelski-2.5}, there exists a torsion-free group $H$ with $|H|=\kappa$ and $|H^2|=\mu$. Let $G=H\times C_2$, $C_2=\langle c\rangle$ with $c$ of order $2$, $a=(e,c)$, $A=\{a\}$, and $B=\{(s,c):s\in H^2,\,s\ne e\}$.
Then $|A|=1<\kappa$ and $|B|\le\mu<\kappa$. Exactly as in Theorem~\ref{thm:converse}, $a$ is a central element of order $2$, $H[2]=\{e\}$, and $W_{G}(A,B)=H[2]\times C_2$ is finite. Hence the pair cannot be separated.
\end{proof}

From this we can give a more general conclusion.

\begin{corollary}\label{cor:successor}
On every successor cardinal $\kappa=\mu^+$ there exists a counterexample.
\end{corollary}

\begin{proof}
Let $\kappa=\mu^+$ be an arbitrary successor cardinal. Then $\mu\ge\aleph_0$ is an infinite cardinal and $\mu<\mu^+=\kappa$. By the standard cardinal inequality for $\Ded\mu$, namely
$\mu^+<\Ded\mu\le(2^\mu)^+$ (as quoted at Newelski \cite{Newelski1987}, Definition 2.4), we get $\kappa=\mu^+<\Ded\mu$. All hypotheses of Theorem~\ref{thm:spectrum} are satisfied, so a counterexample of cardinality exactly $\kappa$ exists.
\end{proof}

Therefore, whether or not CH holds, there are counterexamples on all of $\aleph_1$, $\aleph_2$, $\aleph_3,\ldots$. For example, in a model with $2^{\aleph_0}=\aleph_2$, the cardinal $\aleph_1$ is not of the form $2^\lambda$, yet counterexamples still exist. What Theorem~\ref{thm:spectrum} does not cover are the cardinals satisfying $\Ded\mu\le\kappa$ for all $\mu<\kappa$. But the existing results already suffice to show that there is no essential connection between the size of $G$ and separability, and we shall not pursue this direction further.

\subsection{The size of the $W$-witness set}

We now return to the $W$-witness set. Although the range $\aleph_0 \le |W| \le \lambda$ is still out of reach, and although Proposition~\ref{prop:skeleton} can in principle produce $W$-witness sets of arbitrary size, we notice that the two infinite counterexamples given by Propositions~\ref{prop:sep} and \ref{prop:sep2} both satisfy $|W|=\aleph_0$. Could $2^{\aleph_0}$ then be a critical point for separability? We now push the method of $W$-witness sets further.

\begin{theorem}\label{thm:ceiling}
If $A,B$ cannot be separated, then $|W|\le\lambda<|G|$. Equivalently,
\[
  |W|=|G|\ \Longrightarrow\ A,B\ \text{can be separated}.
\]
\end{theorem}

\begin{proof}
By Lemma~\ref{lem:greedy}, if $|W|>\lambda$ then $A,B$ can be separated, so non-separability forces $|W|\le\lambda$. It remains to show $\lambda<|G|$. If $|G|=\aleph_0$, then $|A|,|B|<\aleph_0$ forces $A,B$ to be finite, and by Theorem~\ref{thm:finite} they can be separated, contradicting the hypothesis; hence $|G|>\aleph_0$. In this case $|A|,|B|,\aleph_0<|G|$, so $\lambda=\max\{|A|,|B|,\aleph_0\}<|G|$.
\end{proof}

Thus $|W|=|G|$ is a sufficient condition for separability, and $2^{\aleph_0}$ has no special status whatsoever. How large, then, can $|W|$ be for a non-separable pair?

We generalize the group of $S_3$-valued functions of Proposition~\ref{prop:sep}. Let $I$ be an index set of cardinality $\lambda$ and let $\rho\le\lambda$ be an infinite cardinal. Put $\tau=(12)$ and define
\[
G_{\lambda,\rho}=\Bigl\{f\in S_3^{I}:\bigl|\{i\in I:f(i)\notin A_3\}\bigr|<\rho\Bigr\}, \quad
E=\Bigl\{f\in G_{\lambda,\rho}:\forall i\in I, f(i)\in\{e,\tau\}, |\supp f|<\rho\Bigr\},
\]
\[
a_i(j) = \begin{cases}
\tau, & j = i, \\
e,   & j \ne i,
\end{cases} \quad A=\{a_i\in S_3^{I}:i\in I\},
\]
\[
  B=\Bigl\{ f\in G_{\lambda,\rho}: \exists i\in I\;\bigl( f(i)\in\{(13),(23)\} \land \forall j\ne i,\; f(j)=e \bigr) \Bigr\}\cup\bigl(E\setminus(A\cup\{e\})\bigr).
\]

\begin{proposition}\label{prop:Glr}
For the above $G_{\lambda,\rho},E,A,B$, writing $\lambda^{<\rho}=\sup_{\theta<\rho}\lambda^\theta$, we have:
\begin{enumerate}
  \item[(1)] $|G_{\lambda,\rho}| = 2^{\lambda}$;
  \item[(2)] $A \cap B = \varnothing$, and $|A| = \lambda$, $|E| = |B| = \lambda^{<\rho}$; in particular, when $\lambda^{<\rho} < 2^{\lambda}$, we have $|A|,|B| < |G_{\lambda,\rho}|$;
  \item[(3)] $W_G(A,B) = E$, so that $|W| = \lambda^{<\rho}$;
  \item[(4)] every set $X \subset G_{\lambda,\rho}$ satisfying (2)(3)(4) of Definition~\ref{def:sep} is necessarily $\{e\}$, and therefore $A$ and $B$ are not separated.
\end{enumerate}
\end{proposition}

\begin{proof}
(1) Since $A_3<S_3$ and the operation is coordinatewise: if $f,g\in G_{\lambda,\rho}$, then whenever $f(i),g(i)\in A_3$ we have $f(i)g(i)\in A_3$, so the set of coordinates where $fg$ lies outside $A_3$ is contained in the union of the corresponding sets for $f$ and for $g$; hence $|\{i\in I: f(i)g(i)\notin A_3\}|<\rho+\rho=\rho$, and inversion clearly preserves the condition. Therefore $G_{\lambda,\rho}$ is a subgroup of $S_3^{I}$. From $\prod_{i\in I} A_3\subset G_{\lambda,\rho}$ we get $|G_{\lambda,\rho}|\ge 3^{\lambda}=2^{\lambda}$; on the other hand, $G_{\lambda,\rho}\subset S_3^{I}$ gives $|G_{\lambda,\rho}|\le 6^{\lambda}=2^{\lambda}$. Hence $|G_{\lambda,\rho}|=2^{\lambda}$.

(2) By the definition of $a_i$, only its $i$-th coordinate takes $\tau$ and all other coordinates take $e$. Clearly $A\cap B=\varnothing$. The map $I\rightarrow A$, $i\mapsto a_i$, is injective, so $|A|=\lambda$. Each $f\in E$ is uniquely determined by its support, so $|E|$ equals the number of subsets of $I$ of cardinality $<\rho$.

For each finite cardinal $n$, $I$ has $\binom{\lambda}{n}$ subsets of size $n$, so the total number of finite subsets is $\sum_{n<\omega}\binom{\lambda}{n}=\lambda$. For each infinite cardinal $\theta<\rho$, $I$ has $\binom{\lambda}{\theta}$ subsets of size $\theta$. Since every $\theta$-element subset is the image of some injective map $\theta\to I$, and the set of injections is contained in the set of functions $I^\theta$, we have $\binom{\lambda}{\theta}\le\lambda^\theta$; conversely, for each $f\in I^\theta$, its graph $\operatorname{graph}(f)=\{(\alpha,f(\alpha)):\alpha<\theta\}$ is a $\theta$-element subset of $\theta\times I$, and $f\mapsto\operatorname{graph}(f)$ is injective, while $|\theta\times I|=\lambda$, so $\lambda^\theta\le\binom{\lambda}{\theta}$. Hence $\binom{\lambda}{\theta}=\lambda^\theta$.

Since $\rho$ is infinite, the infinite-cardinality part contributes $\sum_{\aleph_0\le\theta<\rho}\lambda^\theta=\sup_{\theta<\rho}\lambda^\theta=\lambda^{<\rho}$, and the finite part $\lambda\le\lambda^{\aleph_0}\le\lambda^{<\rho}$ is absorbed. Therefore the number of subsets of $I$ of cardinality $<\rho$ is $\lambda^{<\rho}$, i.e., $|E|=\lambda^{<\rho}$. Then $|B|\le 2\lambda+|E|=\lambda^{<\rho}$. Conversely, $E\setminus(A\cup\{e\})\subset B$; if $\lambda^{<\rho}>\lambda$, then removing from the set $E$ of cardinality $\lambda^{<\rho}$ the subset $A\cup\{e\}$ of cardinality $\lambda$ leaves a set of cardinality still $\lambda^{<\rho}$; if $\lambda^{<\rho}=\lambda$, then $\{f\in E: |\supp(f)|=2\}\subset E\setminus (A\cup \{e\})$ and there are $\binom{\lambda}{2}=\lambda$ functions with support of size exactly $2$, so $|E\setminus (A\cup \{e\})|\ge\lambda=\lambda^{<\rho}$. Since also $|E\setminus (A\cup \{e\})|\le|E|=\lambda^{<\rho}$, we get $|E\setminus (A\cup \{e\})|=\lambda^{<\rho}$. In both cases $|E\setminus(A\cup\{e\})|=\lambda^{<\rho}$, so $|B|\ge |E\setminus(A\cup\{e\})|=\lambda^{<\rho}$; combining with the upper bound gives $|B|=\lambda^{<\rho}$.

(3) Let $g\notin E$. Then there exists $i\in I$ with $g(i)\notin\{e,\tau\}=C_{S_3}(\tau)$. Take $a=a_i\in A$; then the $i$-th coordinate of $gag^{-1}$ is $g(i)\tau g(i)^{-1}$, and every other coordinate is $g(j)e g(j)^{-1}=e$. Since $g(i)\notin C_{S_3}(\tau)$, we have $g(i)\tau g(i)^{-1}\ne\tau$, so this coordinate equals $(13)$ or $(23)$, and $gag^{-1}$ falls into the first part of $B$. Apart from the $i$-th coordinate, all coordinates of $gag^{-1}$ equal $e\in A_3$, so its set of coordinates outside $A_3$ is exactly $\{i\}$, of size $1<\rho$; hence $gag^{-1}\in G_{\lambda,\rho}$. Thus $g\notin W$.

Conversely, let $g\in E$. Then each coordinate $g(j)\in\{e,\tau\}$ and $g(j)^2=e$, so $g^{-1}=g$ and $g$ commutes with $\tau$. For any $a=a_i$ we get $gag=a$, hence $\{g,g^{-1}\}A\{g,g^{-1}\}=gAg=A$, which is disjoint from $B$, so $g\in W$. In summary, $W=E$, and hence $|W|=\lambda^{<\rho}$.

(4) Let $X\subset G_{\lambda,\rho}$ satisfy $e\in X$, $X=X^{-1}$, $XAX\cap B=\varnothing$. By Lemma~\ref{lem:necessary}, $X\subset W=E$. If there exists $x\in X$ with $x\ne e$, then $\supp(x)$ is nonempty with $|\supp(x)|<\rho\le\lambda$, so we may take $i\in I\setminus\supp(x)$. Then
\[
xa_i = xa_ie \in XAX,
\]
so $xa_i\notin B$. For $j\ne i$ we have $xa_i(j)=x(j)\in\{e,\tau\}$; for $j=i$ we have $xa_i(i)=x(i)\tau=\tau$. Hence $\supp(xa_i)=\supp(x)\cup\{i\}$, so still $|\supp(xa_i)|<\rho$, and $xa_i\in E$. Since $|\supp(xa_i)|=|\supp(x)|+1\ge 2$, the element $xa_i$ is neither $e$ nor one of the $a_m$ with a single nontrivial coordinate, so $xa_i\in E\setminus(A\cup\{e\})\subset B$, a contradiction. Hence $X$ has no nontrivial element, and since $e\in X$ we get $X=\{e\}$, a finite set that fails (1) of Definition~\ref{def:sep}. Therefore $A$ and $B$ are not separated.
\end{proof}

With this generalization in hand, different choices of parameters give the corresponding results: taking $(\lambda,\rho)=(\aleph_0,\aleph_0)$ recovers Proposition~\ref{prop:sep}; taking $(\lambda,\rho)=(\aleph_1,\aleph_0)$ yields a non-separable example with $W$ uncountable, where $|W|=\aleph_1$ while the ambient group has cardinality $2^{\aleph_1}>|W|$.

In fact, for any infinite cardinal $\nu$, taking $\lambda=\nu$ and $\rho=\aleph_0$ gives $\lambda^{<\aleph_0}=\sup_{n<\omega}\nu^n=\nu$, and Proposition~\ref{prop:Glr} produces a group of cardinality $|G|=2^\nu$ together with a non-separable pair $(A,B)$ in it satisfying $|A|=|B|=|W|=\nu<2^\nu$. Hence, for every infinite cardinal $\nu$ there exists a non-separable pair with $|W|=\nu$; $|W|$ is not bounded by any fixed cardinal independent of $|G|$, and the restriction $|W|<|G|$ given by Theorem~\ref{thm:ceiling} cannot be replaced by an absolute cardinal bound.

After investigating the influence of the sizes of $G$ and $W$ on separability, we find that they appear to have no essential connection. So there seems to be no way of deciding separability by cardinalities alone.

From the perspective of Question 17.102 itself, one wants to use the sizes of $A$ and $B$ as a criterion for separability. Is there then a criterion for separability, both sufficient and necessary, phrased purely in terms of $|A|,|B|$ and certain structural invariants of $G$ (rather than $W$ itself)?

\subsection{Nonexistence of intrinsic criteria}

\begin{theorem}\label{thm:indis}
In the same group $G=\Sym(\mathbb{N})$ there exist two pairs of disjoint subsets $(A_1,B_1)$, $(A_2,B_2)$ satisfying
\[
  |A_1|=|B_1|=|A_2|=|B_2|=\aleph_0<|G|,
\]
such that $(A_1,B_1)$ cannot be separated while $(A_2,B_2)$ can.
\end{theorem}

\begin{proof}
Take $(A_1,B_1)$ to be the non-separable pair of Example~\ref{ex:ray}. We construct a separable pair $(A_2,B_2)$. For each $n\ge 0$ let
\[
  g_n=(2n\;\,2n+1)\in G.
\]
Clearly the $g_n$ commute pairwise, and $g_n^2=e$. Let $X=\{e\}\cup\{g_n:n\ge 0\}$,
\[
  A_2=\{e\}\cup\{g_{2k}g_{2k+1}:k\ge 0\},\qquad
  B_2=\{(3n\;\,3n+1\;\,3n+2):n\ge 0\}.
\]
Clearly $e\in X$; since $g_n^{-1}=g_n$ we have $X=X^{-1}$, and $X$ is infinite. It is easy to see that $|A_2|=|B_2|=\aleph_0$.

We now verify $XA_2X\cap B_2=\varnothing$. Take any $x_1,x_2\in X$ and $a\in A_2$; then
\[
  x_1ax_2=\prod_{n\in S}g_n,
\]
where $S\subset\mathbb N$ is finite and the right-hand side is a product of $|S|$ pairwise disjoint transpositions, with $S=\varnothing$ interpreted as $e$. If $S\ne\varnothing$, the transpositions being pairwise disjoint, the right-hand side squares to $e$ but is not $e$ itself, so it has order $2$; whereas the $3$-cycles in $B_2$ have order $3$, so the two cannot be equal. If $S=\varnothing$, the right-hand side is $e$, also not a $3$-cycle. Hence $x_1ax_2\notin B_2$. By the arbitrariness of $x_1,x_2,a$, we get $XA_2X\cap B_2=\varnothing$. Moreover, $e\in X$ implies $A_2\subset XA_2X$, so $A_2\cap B_2=\varnothing$. Thus $X$ is an infinite symmetric subset witnessing that $(A_2,B_2)$ is separable.
\end{proof}

Since both pairs of subsets live in the same group $G$, every structural invariant of $G$ takes the same value in the two cases, and $|A|,|B|$ are also the same, while separability differs. Therefore no necessary and sufficient criterion can depend only on $|A|,|B|$ and on the group $G$ itself, without reference to the concrete position of $A,B$ in $G$; separability depends essentially on the way $A,B$ are embedded in the group, which is exactly the information recorded by $W_G(A,B)$.

Thus the $W$-witness set is a rather useful tool, but it cannot resolve all cases. Can we develop a still more useful tool?

\section{The $W$-Witness Graph and a Necessary and Sufficient Condition for Separability}\label{sec:conflict}

It is worth taking a second look at the proof of Lemma~\ref{lem:greedy}. At each step, when a new element $x$ is added, the candidates that get excluded are precisely the solutions of the equation $xay=b$ and its three variants, and the argument works only because the total number of such solutions is less than $|W|$; once $|W|\le\lambda$, the counting breaks down. The equations themselves, however, do not disappear when the counting fails: they record which pairs of elements cannot enter a witness set together, and we call this relation between two elements a \emph{conflict}. The proof used only the total number of conflicts, not which pairs of elements the conflicts actually occur between, and the examples of Section~\ref{sec:alt}, in which $|W|=\aleph_0$ yet $A$ and $B$ are inseparable, show that success or failure hinges precisely on the latter.

We therefore need to study the conflict relation in its own right. If the condition $XAX\cap B=\varnothing$ fails, then there is always some pair $x,x'\in X$ and some $a\in A$ with $xax'\in B$, so ``cannot coexist'' is a binary relation on pairs of elements; moreover, since a witness set is symmetric, $x$ and $x^{-1}$ enter and leave together, so the relation should be carried not by single elements but by inverse pairs $\{g,g^{-1}\}$; as for an element $g\notin W$, by Lemma~\ref{lem:necessary} it does not even qualify to enter a witness set alone, so we may regard it as conflicting with itself. Vertices, conflicts between vertices, and conflicts of a vertex with itself: all three ingredients are present, and a graph is already at hand; it only remains to write it down.

On $G$ define the binary relation
\[
u\sim v\iff v=u\ \text{or}\ v=u^{-1}.
\]
It is an equivalence relation: reflexivity and symmetry are obvious; if $u\sim v$ and $v\sim w$, then $\{w,w^{-1}\}=\{v,v^{-1}\}=\{u,u^{-1}\}$, hence $u\sim w$, and transitivity holds. Denote by $\overline g=\{g,g^{-1}\}$ the class of $g$, and by $V=G/{\sim}=\bigl\{\overline g:g\in G\bigr\}$ the quotient set. When $g^2=e$ the class $\overline g$ is a singleton, in particular $\overline e=\{e\}$; otherwise it is a two-element set.

\begin{definition}\label{def:conflict}
With vertex set $V$ define the graph $\Gamma_{A,B}=(V,E)$, where
\[
E=\bigl\{\{\overline g\}:g\notin W\bigr\}\cup\bigl\{\{\overline u,\overline v\}:\overline u\ne\overline v,\exists a\in A,\exists b\in B,\exists \varepsilon,\delta\in\{\pm1\}, u^{\varepsilon}av^{\delta}=b\bigr\}.
\]
A one-element edge $\{\overline g\}$ is a loop at the vertex $\overline g$, and the rest are undirected edges between two distinct vertices. The definition of the edge set is independent of the choice of representatives: the loop part relies on $W=W^{-1}$, while for the ordinary edges, since $\varepsilon$ and $\delta$ range over $\{\pm1\}$, changing representatives alters neither the solvability nor the unsolvability of the equation. This graph is called the \emph{$W$-witness graph}. Adjacency (loops included) is also called \emph{conflict}, and accordingly $\Gamma_{A,B}$ is also called the \emph{conflict graph} of $A$ and $B$.
\end{definition}

\begin{remark}
Throughout this paper, an independent set $I$ does not contain all endpoints of any edge; in particular, a vertex carrying a loop does not belong to any independent set.
\end{remark}

Once the graph is constructed, it is natural to ask what each of the four axioms of a witness set becomes in graph language. Condition (3) has already been absorbed into the choice of vertices: since vertices are inverse pairs and a union of inverse pairs is automatically closed under inversion, a symmetric $X$ is nothing but a union of vertices. Condition (4) requires $XAX\cap B=\varnothing$, and this fails if and only if some pair $x,x'\in X$ and some $a\in A$ satisfy $xax'\in B$ --- if $x$ and $x'$ lie in distinct vertices, this is precisely an edge between the two vertices; if they lie in the same vertex, it is precisely a loop at that vertex. So (4) says that no edge at all joins these vertices, that is, they form an independent set. Condition (2) says that $\overline e$ belongs to this independent set. As for (1), since the map $x\mapsto\overline x$ is at most two-to-one, $X$ is infinite precisely when the set of vertices it occupies is infinite. Having translated the four axioms one by one, we find that a witness set and an infinite independent set containing $\overline e$ are two descriptions of the same object. We record this as follows.

\begin{theorem}\label{thm:indep}
$A$ and $B$ are separable if and only if $\Gamma_{A,B}$ contains an infinite independent set that includes $\overline e$.
\end{theorem}

\begin{proof}
$(\Rightarrow)$ Let $X$ be a witness set for $A$ and $B$, i.e., $X$ satisfies (1)--(4) of Definition~\ref{def:sep}. Set
\[
I=\{\overline x:x\in X\}\subset V,
\]
the set of vertices occupied by the elements of $X$. We verify that $I$ is an infinite independent set containing $\overline e$. By (2) we have $e\in X$, hence $\overline e\in I$. By (3), (4) and Lemma~\ref{lem:necessary}, $X\subset W$, so no vertex occupied by an element of $X$ carries a loop. Let $u,v\in X$ with $\overline u\ne\overline v$. If $\overline u$ and $\overline v$ were adjacent, then by the definition of the edge set there would exist $a\in A$, $b\in B$ and $\varepsilon,\delta\in\{\pm1\}$ with $u^{\varepsilon}av^{\delta}=b$. By (3), $u^{\varepsilon},v^{\delta}\in X$, so $b=u^{\varepsilon}av^{\delta}\in XAX\cap B$, contradicting (4); hence no two distinct vertices of $I$ are adjacent. If $I$ were finite, then $X=\bigcup_{\overline x\in I}\overline x$ would be a finite union of sets with at most two elements, hence finite, contradicting (1); hence $I$ is infinite. Therefore $I$ is an infinite independent set containing $\overline e$.

$(\Leftarrow)$ Let $I$ be an infinite independent set containing $\overline e$. Set
\[
X=\bigcup_{\overline g\in I}\overline g=\{x\in G:\overline x\in I\},
\]
i.e., expand each equivalence class in $I$ back into elements of $G$. We verify that $X$ satisfies (1)--(4) of Definition~\ref{def:sep}.

(1) Distinct equivalence classes are disjoint, and $X$ meets every class of $I$; since $I$ is infinite, $|X|\ge|I|\ge\aleph_0$, so $X$ is infinite.

(2) Since $\overline e=\{e\}\in I$, we have $e\in X$.

(3) Each class $\overline g=\{g,g^{-1}\}$ is closed under inversion, hence so is their union $X$, i.e., $X=X^{-1}$.

(4) $I$ is independent, so its vertices carry no loops; a loop occurs precisely at the class of an element $g\notin W$, hence $\overline g\in I$ implies $g\in W$, and therefore $X\subset W$. Now suppose $xax'\in B$ for some $x,x'\in X$ and $a\in A$. If $\overline x\ne\overline{x'}$, the equality $x^{+1}a(x')^{+1}\in B$ is exactly an instance of the edge definition, showing that $\overline x$ and $\overline{x'}$ are two adjacent vertices of $I$, contradicting independence. If $\overline x=\overline{x'}$, then $x'\in\{x,x^{-1}\}$, so $xax'\in\{x,x^{-1}\}A\{x,x^{-1}\}$; but $x\in X\subset W$, and by the definition of $W$, $\{x,x^{-1}\}A\{x,x^{-1}\}\cap B=\varnothing$, contradicting $xax'\in B$. Neither case is possible, hence $XAX\cap B=\varnothing$.

Thus $X$ is a witness set, and $A$ and $B$ are separable.
\end{proof}

We next ask which vertices an independent set containing $\overline e$ can actually use. The following definition answers this.

\begin{definition}\label{def:S}
Set
\[
S_{A,B}=\bigl\{\overline g\in V:g\in W,\ \overline g\ne\overline e,\ \overline g\text{ is not adjacent to }\overline e\bigr\},
\]
that is, the set of all vertices represented by elements of $W$ that differ from $\overline e$ and do not conflict with $\overline e$; write $\Gamma_{A,B}[S_{A,B}]$ for the subgraph induced by $S_{A,B}$ in $\Gamma_{A,B}$.
\end{definition}

It is clear that $|S_{A,B}|\le|W_G(A,B)|$.

On the one hand, a looped vertex belongs to no independent set, and a vertex adjacent to $\overline e$ cannot share an independent set with $\overline e$, so an independent set containing $\overline e$ can draw on nothing but $S_{A,B}$ beyond $\overline e$ itself; on the other hand, $\overline e$ itself carries no loop (since $e\in W$), and by definition no vertex of $S_{A,B}$ is adjacent to $\overline e$. Putting the two sides together yields another necessary and sufficient condition for separability.

\begin{theorem}\label{thm:indep-S}
$A$ and $B$ are separable if and only if the induced subgraph $\Gamma_{A,B}[S_{A,B}]$ contains an infinite independent set.
\end{theorem}

\begin{proof}
$(\Rightarrow)$ Let $X$ be a witness set. The set $I=\{\overline x:x\in X\}$ constructed in the proof of Theorem~\ref{thm:indep} is an infinite independent set containing $\overline e$, so $I\setminus\{\overline e\}$ is still infinite. Its vertices carry no loops, i.e., their representatives lie in $W$; since $I$ is independent, none of them is adjacent to $\overline e$, and all differ from $\overline e$. Hence $I\setminus\{\overline e\}\subset S_{A,B}$, and it is an infinite independent set in $\Gamma_{A,B}[S_{A,B}]$.

$(\Leftarrow)$ Let $J$ be an infinite independent set in $\Gamma_{A,B}[S_{A,B}]$. By the definition of an induced subgraph, two vertices of $S_{A,B}$ are adjacent in the induced subgraph precisely when they are adjacent in $\Gamma_{A,B}$, so any two distinct vertices of $J$ are non-adjacent in the whole graph as well. By the definition of $S_{A,B}$, no vertex of $J$ carries a loop, and none is adjacent to $\overline e$; while $\overline e$ itself carries no loop. Hence $J\cup\{\overline e\}$ is an infinite independent set containing $\overline e$, and by Theorem~\ref{thm:indep}, $A$ and $B$ are separable.
\end{proof}

In fact, the proof of Theorem~\ref{thm:indep-S} yields a cardinality-preserving correspondence between witness sets and independent sets. For an infinite cardinal $\kappa$, call $A$ and $B$ \emph{$\kappa$-separable} if $A$ and $B$ admit a witness set of cardinality at least $\kappa$.

\begin{corollary}\label{cor:kappa-sep}
For every infinite cardinal $\kappa$, $A$ and $B$ are $\kappa$-separable if and only if $\Gamma_{A,B}[S_{A,B}]$ contains an independent set of cardinality at least $\kappa$. Let
\[
\alpha=\sup\{|X|:X\text{ is a witness set for }A,B\}, \quad
\beta=\sup\{|I|:I\text{ is an independent set in }\Gamma_{A,B}[S_{A,B}]\};
\]
then $\alpha$ and $\beta$ are simultaneously finite or simultaneously infinite, and in the latter case $\alpha=\beta$.
\end{corollary}

\begin{proof}
$(\Rightarrow)$ Let $X$ be a witness set of cardinality at least $\kappa$. In the proof of Theorem~\ref{thm:indep-S}, the set $I\setminus\{\overline e\}=\{\overline x:x\in X,\ x\ne e\}$ is an independent set in $\Gamma_{A,B}[S_{A,B}]$. The map $x\mapsto\overline x$ is at most two-to-one, so $|I\setminus\{\overline e\}|\ge|X|/2=\kappa$.

$(\Leftarrow)$ Let $J$ be an independent set in $\Gamma_{A,B}[S_{A,B}]$ of cardinality at least $\kappa$. Set $X=\{e\}\cup\bigcup_{\overline g\in J}\overline g$. By the argument in the $(\Leftarrow)$ direction of Theorem~\ref{thm:indep}, $X$ satisfies (2), (3) and (4) of Definition~\ref{def:sep}. Distinct equivalence classes are pairwise disjoint and each $\overline g$ contains at least one element, so $\left|\bigcup_{\overline g\in J}\overline g\right|\ge|J|$, and therefore $|X|\ge1+|J|=|J|\ge\kappa$.

Finally, the assertion about the suprema. If $\alpha$ is infinite, then for every infinite cardinal $\kappa<\alpha$ there is a witness set $X$ with $|X|\ge\kappa$; by the $(\Rightarrow)$ direction of this corollary there is an independent set of cardinality at least $\kappa$, so $\beta\ge\kappa$. Letting $\kappa$ range over all infinite cardinals below $\alpha$ gives $\beta\ge\alpha$, so $\beta$ is infinite as well. Conversely, if $\beta$ is infinite, the same argument via the $(\Leftarrow)$ direction gives $\alpha\ge\beta$, so $\alpha$ is infinite too. Thus $\alpha$ and $\beta$ are simultaneously finite or simultaneously infinite, and in the latter case both inequalities above hold, whence $\alpha=\beta$.
\end{proof}

By Theorem~\ref{thm:indep-S}, separability reduces to a plain infinite-independent-set problem on a single induced subgraph, and we can immediately record a sufficient condition phrased in terms of the chromatic number.

\begin{corollary}\label{cor:chromatic}
If $S_{A,B}$ is infinite and the chromatic number of $\Gamma_{A,B}[S_{A,B}]$ is finite, then $A$ and $B$ are separable. More generally, if $S_{A,B}$ is a union of $\mu$ independent sets with $\mu<\cf|S_{A,B}|$, then $A$ and $B$ are separable.
\end{corollary}

\begin{proof}
A finite coloring partitions $S_{A,B}$ into finitely many independent sets; a finite union of finite sets is finite, while $S_{A,B}$ is infinite, so at least one color class is infinite, and Theorem~\ref{thm:indep-S} gives separability.

The cardinal case is analogous. Let $\kappa=|S_{A,B}|$ and write $S_{A,B}=\bigcup_{i<\mu}C_i$, where each $C_i$ is independent and $\mu<\cf\kappa$. Suppose toward a contradiction that $|C_i|<\kappa$ for every $i<\mu$. Then
\[
\kappa = |S_{A,B}| = \left|\bigcup_{i<\mu}C_i\right| \le \sum_{i<\mu}|C_i| \le \mu\cdot\sup_{i<\mu}|C_i|.
\]
By the definition of cofinality, the supremum of fewer than $\cf\kappa$ cardinals each smaller than $\kappa$ is still strictly less than $\kappa$, so $\sup_{i<\mu}|C_i|<\kappa$. Moreover $\mu<\cf\kappa\le\kappa$, and the product of two infinite cardinals smaller than $\kappa$ equals the larger of the two, so
\[
\mu\cdot\sup_{i<\mu}|C_i| = \max\{\mu,\sup_{i<\mu}|C_i|\} < \kappa,
\]
which together with the displayed inequality gives $\kappa<\kappa$, a contradiction. Hence $|C_i|=\kappa$ for some $i<\mu$, and this $C_i$ is an infinite independent set, so Theorem~\ref{thm:indep-S} yields separability.
\end{proof}

Theorem~\ref{thm:indep} translates separability into graph language, but whether the whole graph admits an infinite independent set containing $\overline e$ is a global question, with no obvious direct approach. A natural starting point is local information: if the degrees of all vertices admit a uniform upper bound, then when building an independent set vertex by vertex, each step rules out at most boundedly many vertices, and whether one can keep going depends on a single variable --- the total number of candidate vertices. We therefore seek an upper bound for the vertex degrees.

\begin{lemma}\label{lem:deg}
Apart from loops, the degree of every vertex is at most $4|A||B|$.
\end{lemma}

\begin{proof}
Fix a vertex $\overline u$ and a representative $u\in\overline u$; what is to be estimated is the number of distinct vertices $\overline v\ne\overline u$ adjacent to $\overline u$. By definition, $\overline u$ and $\overline v$ are adjacent if and only if there exist $a\in A$, $b\in B$ and $\varepsilon,\delta\in\{\pm1\}$ with $u^{\varepsilon}av^{\delta}=b$, where $v$ is some representative of $\overline v$; since $\varepsilon,\delta$ range over $\pm1$, the choice of representatives does not affect adjacency, so we may count with $u$ fixed, one parameter tuple at a time.

For each fixed tuple $(a,b,\varepsilon,\delta)\in A\times B\times\{\pm1\}\times\{\pm1\}$, view $u^{\varepsilon}av^{\delta}=b$ as an equation in $v$. Rearranging gives $v^{\delta}=a^{-1}u^{-\varepsilon}b$: if $\delta=+1$ then $v=a^{-1}u^{-\varepsilon}b$; if $\delta=-1$, inverting both sides gives $v=b^{-1}u^{\varepsilon}a$. In both cases $v$ is uniquely determined by the tuple, i.e., each tuple solves for at most one $v\in G$, and hence contributes at most one neighbor $\overline v$.

There are $|A|\cdot|B|\cdot2\cdot2=4|A||B|$ tuples, so the total number of solutions $v$ is at most $4|A||B|$, and hence so is the number of neighbors. In fact the bound may be smaller still: different tuples may yield the same $v$, and a solution $v$ lying in $\overline u$ itself produces no ordinary edge by definition. Hence the degree of $\overline u$ (not counting loops) is at most $4|A||B|$.
\end{proof}

With the degree bound of Lemma~\ref{lem:deg} in hand, the old cardinal criterion becomes a corollary of Theorem~\ref{thm:indep}: when the loop-free vertices (i.e., the classes of elements of $W$) number more than $4|A||B|$ --- the case $|W|>\lambda$ of Lemma~\ref{lem:greedy} --- the vertex-by-vertex construction must produce an infinite independent set; and when $A$ and $B$ are finite, the graph is locally finite, so an infinite independent set exists as well. Within the gap $\aleph_0\le|W|\le\lambda$, however, counting no longer suffices, and the verdict is decided by the actual structure of the graph --- the two examples of Section~\ref{sec:alt} with $|W|=\aleph_0$ are precisely extremal: every loop-free vertex is adjacent to $\overline e$, and the only independent set containing $\overline e$ is the single vertex $\{\overline e\}$.

Reconsidering $W$ in graph-theoretic terms, we also obtain two further sufficient conditions for inseparability:

\begin{corollary}\label{cor:lock}
If $\overline e$ is adjacent to every other loop-free vertex, i.e., $\overline w$ is adjacent to $\overline e$ for every $w\in W\setminus\{e\}$, then the only independent set containing $\overline e$ is $\{\overline e\}$, and $A$ and $B$ are inseparable; more precisely, any $X$ satisfying (2), (3) and (4) must be $\{e\}$.
\end{corollary}

\begin{proof}
Let $I$ be an independent set containing $\overline e$. The vertices of $I$ carry no loops, while every loop-free vertex other than $\overline e$ is adjacent to $\overline e$ and thus cannot enter $I$; hence $I=\{\overline e\}$, and by Theorem~\ref{thm:indep}, $A$ and $B$ are inseparable. If $X$ satisfies (2), (3) and (4), the construction in the $(\Rightarrow)$ direction of Theorem~\ref{thm:indep} produces an independent set $\{\overline x:x\in X\}$ containing $\overline e$, which can only be $\{\overline e\}$; hence $X=\{e\}$.
\end{proof}

\begin{corollary}\label{cor:unilateral}
If for every $g\in W\setminus\{e\}$ there exist $a\in A$ and $b\in B$ such that
\[
ga=b\qquad\text{or}\qquad g^{-1}a=b,
\]
then $A$ and $B$ are inseparable.
\end{corollary}

\begin{proof}
It suffices to verify the hypothesis of Corollary~\ref{cor:lock}, namely that every loop-free vertex $\overline g\ne\overline e$ is adjacent to $\overline e$. The loop-free vertices are precisely the classes with $g\in W$; fix $g\in W\setminus\{e\}$. Then $g\ne e$, so $\overline g\ne\overline e$. By hypothesis there exist $a\in A$, $b\in B$ with $ga=b$ or $g^{-1}a=b$. Both cases fall under the edge definition $u^{\varepsilon}av^{\delta}=b$: taking $u=g$ and $v=e$, in the first case take $\varepsilon=+1$ to get $g^{+1}ae^{\delta}=ga=b$, and in the second take $\varepsilon=-1$ to get $g^{-1}ae^{\delta}=g^{-1}a=b$. Hence $\overline g$ is adjacent to $\overline e$. By Corollary~\ref{cor:lock}, $A$ and $B$ are inseparable.
\end{proof}

\begin{theorem}\label{thm:ramsey}
If $A$ and $B$ are inseparable and $W$ is infinite, then at least one of the following holds:
\begin{enumerate}
\item[(i)] there are only finitely many $g\in W\setminus\{e\}$ such that $\overline g$ is not adjacent to $\overline e$, i.e., $S_{A,B}$ is finite. Equivalently, there are only finitely many $g\in W\setminus\{e\}$ such that no $a\in A$ and $b\in B$ satisfy $ga=b$ or $g^{-1}a=b$;
\item[(ii)] the induced subgraph $\Gamma_{A,B}[S_{A,B}]$ contains an infinite clique; equivalently, there is an infinite subset $Y$ of $W$ such that for any distinct $u,v\in Y$ there exist $a\in A$, $b\in B$ and $\varepsilon,\delta\in\{\pm1\}$ with $u^{\varepsilon}av^{\delta}=b$.
\end{enumerate}
\end{theorem}

\begin{proof}
By Theorem~\ref{thm:indep-S}, inseparability is equivalent to $\Gamma_{A,B}[S_{A,B}]$ containing no infinite independent set.

\emph{Case 1: $S_{A,B}$ is finite.} This is exactly (i).

\emph{Case 2: $S_{A,B}$ is infinite.} Color each pair of distinct vertices of $S_{A,B}$ red or blue according as they are adjacent or non-adjacent in the induced subgraph. By Theorem~\ref{thm:ramsey-A} (with $r=\mu=2$), there is an infinite monochromatic set. It cannot be blue, for a blue set would be an infinite independent set, contradicting inseparability; hence it is an infinite red set, i.e., an infinite clique in $\Gamma_{A,B}[S_{A,B}]$. Taking one representative from each class of the clique yields an infinite $Y\subset W$ such that the classes of any two distinct $u,v\in Y$ are adjacent, which by the edge definition means that there exist $a\in A$, $b\in B$ and $\varepsilon,\delta\in\{\pm1\}$ with $u^{\varepsilon}av^{\delta}=b$. This gives (ii).
\end{proof}

\begin{remark}
When $W$ is finite, inseparability is trivial; Theorem~\ref{thm:ramsey} addresses the case of infinite $W$. It says that the graph-theoretic reasons for inseparability are of only two kinds: either the supply of usable vertices $S_{A,B}$ is itself finite, or $S_{A,B}$ is infinite and harbors an infinite clique --- there is no third structure. Clause (i) alone already implies inseparability --- an independent set containing $\overline e$ can draw on nothing but the finite $S_{A,B}$ beyond $\overline e$ itself, hence is finite --- so it is a strengthening of Corollary~\ref{cor:lock}; clause (ii), by contrast, is only a necessary condition, since a graph with an infinite clique may also contain an infinite independent set. Question~\ref{q:branch} in Section~\ref{sec:open} will ask whether clause (ii) actually occurs.
\end{remark}

One caveat: Theorem~\ref{thm:indep} does not mean that an infinite independent set can always be found, for an infinite graph need not contain one. What Theorem~\ref{thm:indep-S} does is to remove the extra requirement that the independent set contain the marked vertex $\overline e$: the vertices that could obstruct an independent set --- the looped ones and those adjacent to $\overline e$ --- have already been excluded from $S_{A,B}$ in advance, so separability is ultimately equivalent to the unmarked, ordinary graph-theoretic statement that the induced subgraph $\Gamma_{A,B}[S_{A,B}]$ contains an infinite independent set. The significance of this reformulation is the following. The set $W$ only tests elements one at a time for eligibility to enter a witness set, whereas separability demands that infinitely many eligible elements appear simultaneously --- a collective property. The conflict graph draws the relation ``two elements cannot enter a witness set together'' directly as edges, so that this collective-level incompatibility can be expressed and detected in the language of independent sets, and graph-theoretic tools such as the chromatic number become available (Corollary~\ref{cor:chromatic}). We may therefore bring the many tools of graph theory to bear on the problem.

\section{Some Open Questions}\label{sec:open}

Theorem~\ref{thm:indep-S} reduces separability entirely to the infinite-independent-set problem on the induced graph $\Gamma_{A,B}[S_{A,B}]$, which naturally gives rise to several questions for further study. We record three here, concerning respectively structure, group classes, and set theory.

\begin{question}\label{q:branch}
Does clause (ii) of Theorem~\ref{thm:ramsey} actually occur? In other words, is there an inseparable pair $(A,B)$ with $S_{A,B}$ infinite?
\end{question}

This question is meaningful because all the counterexamples constructed in this paper fall on the side of finite $S_{A,B}$. The following proposition verifies this item by item.

\begin{proposition}\label{prop:all-S-finite}
Every family of inseparable pairs constructed in this paper has finite $S_{A,B}$. Specifically:
\begin{enumerate}
\item the inseparable pairs of Theorem~\ref{thm:l2} have $S_{A,B}$ finite;
\item the inseparable pairs of Theorem~\ref{thm:converse} and of its cardinal generalization, Theorem~\ref{thm:spectrum}, have $S_{A,B}$ finite;
\item the inseparable pairs of Proposition~\ref{prop:skeleton} and of its cardinal generalization, Proposition~\ref{prop:Glr}, have $S_{A,B}=\varnothing$.
\end{enumerate}
\end{proposition}

\begin{proof}
(1) By Proposition~\ref{prop:W-l2}, the set $W_G(A,B)$ in Theorem~\ref{thm:l2} is finite. By the definition of $S_{A,B}$, $|S_{A,B}|\le|W_G(A,B)|$, and hence $S_{A,B}$ is finite.

(2) By Proposition~\ref{prop:W-converse}, in Theorem~\ref{thm:converse} the set $W=H[2]\times C_2$ is finite; in Theorem~\ref{thm:spectrum} likewise $W=H[2]\times C_2$ is finite. In both cases $|S_{A,B}|\le|W|$ yields that $S_{A,B}$ is finite.

(3) Let $(A,B)$ satisfy the hypotheses of Proposition~\ref{prop:skeleton}. The proof of that proposition already yields $W_G(A,B)=\iota(V)$. Take any nontrivial $x=\iota(v)\in W\setminus\{e\}$; then $v\ne0$, and since $A$ is infinite we may pick $a_0\in A\setminus\supp(v)$. By condition (i), $\iota(V)\setminus(\{e\}\cup A)\subset B$, and $\supp(v+e_{a_0})=\supp(v)\cup\{a_0\}$ contains at least two elements, so $\iota(v+e_{a_0})=xa_0\in B$. This is exactly the condition for $\overline x$ and $\overline e$ to be adjacent in the edge definition (take $u=x$, $v=e$, $\varepsilon=+1$, $\delta=+1$). Hence every nontrivial $x\in W$ is adjacent to $\overline e$, i.e., $S_{A,B}=\varnothing$. Proposition~\ref{prop:Glr} is the cardinal generalization of Proposition~\ref{prop:skeleton}; step (4) of its proof likewise constructs elements $a_i\in A$ with $xa_i\in B$, so again $S_{A,B}=\varnothing$.
\end{proof}

If the answer to Question~\ref{q:branch} is negative, then Theorem~\ref{thm:indep-S} yields a remarkably clean criterion: under the assumption that $W$ is infinite,
\[
A\text{ and }B\text{ are inseparable}\quad\Longleftrightarrow\quad S_{A,B}\text{ is finite},
\]
where the ``$\Leftarrow$'' direction holds without any hypothesis, since a finite graph contains no infinite independent set, and the ``$\Rightarrow$'' direction is precisely the literal content of a negative answer to Question~\ref{q:branch}. When $W$ is finite, inseparability is trivial, so the substance of this criterion lies in the half where $W$ is infinite. Thus an affirmative answer to Question~\ref{q:branch} is equivalent to constructing a pair of subsets such that $S_{A,B}$ is infinite while $\Gamma_{A,B}[S_{A,B}]$ contains no infinite independent set. By Theorem~\ref{thm:ramsey-A}, such a graph must contain an infinite clique, so the most natural such construction is one in which every two vertices of $S_{A,B}$ conflict, i.e., the induced graph is an infinite complete graph: every eligible multiplier can enter a witness set alone, but no two of them can enter together.

\begin{question}\label{q:groups}
In FC-groups, nilpotent groups, or solvable groups, is every pair of disjoint subsets with $|A|,|B|<|G|$ separable?
\end{question}

Theorem~\ref{thm:abel} already settles the most special case, that of abelian groups, while both families of counterexamples are far removed from these group classes: the first rests on the rigidity of conjugation in $\Sym(\mathbb N)$, and the second exploits a central element of order two together with a special property of the cardinality of the square. In particular, the second construction requires the group $H$ to satisfy $|H^2|<|H|$ with $H[2]$ finite; if one could prove that some class of groups (e.g., torsion-free FC-groups) cannot satisfy these two conditions, constructions of this kind would be ruled out for that class, although this alone would not yet yield an affirmative answer. Progress in either direction on Question~\ref{q:groups} would push the positive results of Section~\ref{sec:further} substantially beyond abelian groups.

\begin{question}\label{q:singular}
Do counterexample groups exist at singular strong limit cardinals?
\end{question}

Theorem~\ref{thm:spectrum} shows that a counterexample group of cardinality $\kappa$ exists whenever an infinite cardinal $\kappa$ satisfies $\mu<\kappa<\Ded\mu$ for some cardinal $\mu$; successor cardinals are therefore all covered. If, however, $\kappa$ is a singular strong limit cardinal, then for every $\mu<\kappa$ the strong limit property gives $2^\mu<\kappa$; since $\kappa$ is a limit cardinal, it cannot equal the successor of any cardinal, so $(2^\mu)^+<\kappa$, whence $\Ded\mu\le(2^\mu)^+<\kappa$, and the hypothesis of Theorem~\ref{thm:spectrum} fails for every $\mu<\kappa$. Under GCH, $\aleph_\omega$ is the first such cardinal. At such cardinals the square-set construction ceases to work; Question~\ref{q:singular} may require entirely new constructions, and may also be tied to the set-theoretic behavior of $\Ded\mu$ and thus be independent of ZFC.

\section*{A Note on the Use of Artificial Intelligence Tools}

The main mathematical content of this paper --- including the introduction of the $W$-witness set, the construction of the conflict graph, the necessary and sufficient characterization of separability, and the nonexistence of intrinsic criteria --- was carried out independently by the author. The following items were obtained with the help of artificial intelligence tools and are hereby acknowledged:

\begin{itemize}
\item In Section~\ref{sec:rigid}, the construction and proofs of the first family of counterexamples were given independently by the author; within this family, the two examples of the hanging-paths tree and the ray graph were contributed afterwards by GPT-5.6 Luna as further examples of the same type.
\item In Section~\ref{sec:square}, the construction and proofs of the second family of counterexamples were completed independently by the author; the Newelski system of groups used there was found by the author in a discussion with the model Deepseek v4-flash.
\item In Section~\ref{sec:alt}, the two families of counterexamples --- the group of finitely supported functions valued in $S_3$ and the semidirect product $V\rtimes U$ --- together with their proofs, were generated by the model Tencent Hy3 under the author's guidance.
\item The more general cardinal version in Corollary~\ref{cor:chromatic}, stating that $A$ and $B$ are separable whenever $S_{A,B}$ is a union of $\mu$ independent sets with $\mu<\cf|S_{A,B}|$, was proposed by the model Deepseek v4-flash used in the course of writing this paper.
\end{itemize}

All of the above has been independently verified by the author, who bears responsibility for the correctness of the arguments involved.


\begin{thebibliography}{99}

\bibitem{kourovka}
E. I. Khukhro and V. D. Mazurov, editors,
\emph{The Kourovka Notebook: Unsolved Problems in Group Theory},
No.~21, Sobolev Institute of Mathematics, Russian Academy of Sciences,
Novosibirsk, 2026.

\bibitem{Newelski1987}
L. Newelski,
\emph{On the number of squares in a group},
Proc. Amer. Math. Soc. \textbf{99} (1987), no.~2, 213--218.

\bibitem{protasovSurvey}
I. Protasov and K. Protasova,
\emph{Recent progress in subset combinatorics of groups},
J. Math. Sci. \textbf{234} (2018), no.~1, 49--60.

\bibitem{protasov2011}
I.~V. Protasov,
\emph{Selective survey on Subset Combinatorics of Groups},
J. Math. Sci. \textbf{174} (2011), no.~4, 486--514.

\bibitem{partitions2015}
I. Protasov and S. Slobodianiuk,
\emph{Partitions of Groups},
Adv. Appl. Discrete Math. \textbf{15} (2015), no.~1, 33--60.

\bibitem{protasov2003}
I.~V. Protasov,
\emph{Partitions of groups into large subsets},
Math. Notes \textbf{73} (2003), no.~3, 443--446.

\bibitem{protasov2015note}
I. Protasov and S. Slobodianiuk,
\emph{A note on partitions of groups},
Questions Answers Gen. Topology \textbf{33} (2015), 61--70.

\bibitem{banakh2014}
T. Banakh, I. Protasov and S. Slobodianiuk,
\emph{Densities, submeasures and partitions of groups},
Algebra Discrete Math. \textbf{17} (2014), no.~2, 193--221.

\bibitem{zariski2010}
T. Banakh and I. Protasov,
\emph{Zariski topologies on groups},
arXiv:1001.0601 (2010).

\bibitem{Ramsey1930}
F. P. Ramsey,
\emph{On a problem of formal logic},
Proc. London Math. Soc. (2) \textbf{30} (1930), 264--286.

\end{thebibliography}
\end{document}